\documentclass{amsart}
\usepackage{amssymb}
\usepackage{amscd}
\usepackage{multicol}
\usepackage{tikz}
\usepackage{graphics}
\usepackage{shuffle}
\usepackage[OT2,T1]{fontenc}
\usepackage[all]{xy}
\usepackage{extarrows}
\usepackage{comment}
\DeclareSymbolFont{cyrletters}{OT2}{wncyr}{m}{n}
\DeclareMathSymbol{\Sha}{\mathalpha}{cyrletters}{"58}

\usepackage{ulem}

\title[Shuffle product of desingularized multiple zeta functions]
{Shuffle product of desingularized multiple zeta functions at integer points}
\author{Nao Komiyama}
\author{Takeshi Shinohara}

\address{Graduate School of Mathematics, Nagoya University, 
Furo-cho, Chikusa-ku, Nagoya 464-8602 Japan }
\email{}
\date{\today}

\newtheorem{thm}{Theorem}[section]
\newtheorem{lem}[thm]{Lemma}

\newtheorem{prop}[thm]{Proposition}  

\theoremstyle{remark}
        

\numberwithin{equation}{section}
\theoremstyle{definition}
\newtheorem{defn}[thm]{Definition}
\newtheorem{rem}[thm]{Remark}

\newtheorem{exa}[thm]{Examples}

\newcommand{\Li}{{\rm Li}}
\newcommand{\desLi}{\Li^{\rm des}}
\newcommand{\des}{{\rm des}}
\newcommand{\deszeta}{\zeta^{\rm des}}
\newcommand{\hatdeszeta}{\widehat{\zeta}^{\des}}

\newcommand{\veck}{{\bf k}}
\newcommand{\vecl}{{\bf l}}
\newcommand{\vecm}{{\bf m}}
\newcommand{\vecs}{{\bf s}}
\newcommand{\almr}{a_{\vecl,\vecm}^r}

\newcommand{\N}{\mathbb N}
\newcommand{\Z}{\mathbb Z}
\newcommand{\Q}{\mathbb Q}
\newcommand{\R}{\mathbb R}
\newcommand{\C}{\mathbb C}

\newcommand{\otimessym}{\otimes_{\rm sym}}

\newcommand{\dep}{{\rm dep}}
\newcommand{\Id}{{\rm Id}}
\newcommand{\re}{{\rm Re}}

\newcommand{\mcalc}{\mathcal{C}}
\newcommand{\mcalg}{\mathcal{G}}
\newcommand{\mcaln}{\mathcal{N}}
\newcommand{\mcals}{\mathcal{S}}

\begin{document}
\bibliographystyle{amsalpha+}
\maketitle

\begin{abstract}      
  In this paper, we investigate the ``shuffle-type'' formula for special values of 
  desingularized multiple zeta functions at integer points.
  It is proved by giving an iterated integral/differential expression for the desingularized multiple 
  zeta functions at integer points.
\end{abstract}

\tableofcontents
\setcounter{section}{-1}
\section{Introduction}
We begin with the {\it desingularized multiple zeta function} (desingularized MZF for short)
which was introduced by Furusho, Komori, Matsumoto, and Tsumura (\cite{FKMT1}):
\begin{equation}\label{eqn: definition of desingularized MZF}
  \deszeta_r(s_1,\dots,s_r) 
   = \lim_{ \substack{ c \rightarrow 1 \\ c\ne 1 } }
      \frac{1}{(1-c)^r}
       \prod_{k=1}^{r}\frac{1}{(e^{2\pi i s_k}-1)\Gamma(s_k)}
        \int_{\mcalc^r}
         \tilde{\mathfrak{H}}_r(t_1,\dots,t_r;c)
          \prod_{k=1}^{r}t_k^{s_{k}-1}dt_k
\end{equation}
for complex variables $s_1,\dots,s_r$, 
where $\mcalc$ is the {\it Hankel contour} (see Definition \ref{def:definition of des mzf}), and 
for $c\in\R$, we put
\begin{equation*}
 \begin{split}
  \tilde{\mathfrak{H}}_r(t_1,\dots,t_r;c)
   &:= \prod_{j=1}^{r}
            \left(
               \frac{1}{\exp(\sum_{k=j}^{r}t_k)-1}
                - \frac{c}{\exp(c\sum_{k=j}^{r}t_k)-1}
            \right)\ \in\ \C[[t_1,\dots,t_r]]. 
 \end{split}
\end{equation*}
The desingularized MZF was introduced 
to resolve the infinitely many singularities of the {\it multiple zeta function} (MZF for short), 
under the motivation of finding a suitable meaning of the special values of MZF at non-positive integer points.
Refer to 
\S\ref{subsec: The multiple zeta functions} for more information on MZFs and desingularized ones.
Here, we will describe some properties of the desingularized MZFs that are shown in \cite{FKMT1}.
\begin{enumerate}
  \item[(i)] $\deszeta_r(s_1,\dots,s_r)$ is entire on whole space $\C^r$.
  \item[(ii)] $\deszeta_r(s_1,\dots,s_r)$ is expressed as a finite ``linear'' combination of MZFs.
  \item[(iii)] Special values of $\deszeta_r(s_1,\dots,s_r)$ at all-non-positive integer points 
                      are calculated explicitly by using Seki-Bernoulli numbers.
\end{enumerate}

Furthermore, the first-named author showed the following:

\bigskip
\noindent
\textbf{Proposition \ref{prop:shuffle type formula of deszeta}}{(cf. \cite[Theorem 2.7]{Komi3})\textbf{.}}
\textit{
  For $s_1,\dots,s_{p}\in\Z$, $l_1,\dots,l_q\in\N_0$, we have
  \begin{equation*}\label{eqn:recurrence relation of deszeta}
   \begin{split}
    &\deszeta_{p}(s_1,\dots,s_{p})\deszeta_q(-l_1,\dots,-l_q) \\
     &=\sum_{
        \substack{i_b+j_b=l_b   \\   i_b,j_b\ge0   \\   1\le b \le q}
                 }
         \prod_{a=1}^q
          (-1)^{i_a}      
           \binom{l_a}{i_a}
           \deszeta_{p+q}(s_1,\dots,s_{p-1},s_p-i_1-\cdots-i_q,-j_1,\dots,-j_q).
    \end{split}
  \end{equation*}
}
\bigskip

In this paper, 
we examine the shuffle type formula for special values of desingularized MZFs at ``integer'' points.
As a generalization of Proposition \ref{prop:shuffle type formula of deszeta} and the main theorem of this paper, we show the following:

\bigskip
\noindent
\textbf{Theorem \ref{cor:shuffle product of deszeta}.}
\textit{
The ``shuffle-type'' formula holds for special values at any integer points of desingularized MZFs.
}
\bigskip

The following proposition is key to proving this theorem.

\medskip
\noindent
\textbf{Theorem \ref{thm:holomorphic of desLi}.}
\textit{
For $\veck\in\Z^r$, we have
\[
    \desLi(\veck)(t) = \hatdeszeta_r(\veck)(t).
\]
}
\medskip

Here, $\desLi(\veck)(t)$ is defined by a certain iterated integral (Definition \ref{def:definition of L}) and $\hatdeszeta_r(\veck)(t)$ is the special values of a certain function $\hatdeszeta_r(\vecs)(t)$ at integer points defined by the Hankel contour integral 
(Definition \ref{eqn: the definition of hyperposed zeta function}).
By the definition, we know that $\desLi(\veck)(t)$ satisfies the shuffle product formula, and we show that, in the limit as $t\rightarrow1$, the function $\hatdeszeta_r(\vecs)(t)$ coincides with the desingularized MZF $\deszeta_r(\vecs)$ (Proposition \ref{prop:holomorphic of t-deszeta}).
Therefore, we attain our main theorem.

In order to prove Theorem \ref{thm:holomorphic of desLi}, 
we show that both $\desLi(\veck)(t)$ and $\hatdeszeta_r(\veck)(t)$ share the same 
representation.
More precisely, we show the following two:

\bigskip
\noindent
\textbf{Theorem \ref{thm:M(veck)=hypLi}.}
\textit{
  For $\veck\in\Z^r$, we have
  \begin{equation}
    \desLi(\veck)(t) 
       = \sum_{q\ge0}(-1)^q\frac{(\log t)^q}{q!}
            \sum_{\substack{\vecl=(l_j)\in\N_0^r \\ \vecm=(m_j)\in\Z^r \\ |\vecm|=-q}}
              a_{\vecl,\vecm}^r(q)
                \left(  \prod_{j=1}^r(k_j)_{l_j}  \right)
                  \Li_{\veck+\vecm}(t).
  \end{equation}  
}
\medskip

\medskip
\noindent
\textbf{Theorem \ref{thm:hypzeta=M(veck)}.}
\textit{
  For $\veck\in\Z^r$, we have
  \begin{equation}
    \hatdeszeta_r(\veck)(t) 
       = \sum_{q\ge0}(-1)^q\frac{(\log t)^q}{q!}
            \sum_{\substack{\vecl=(l_j)\in\N_0^r \\ \vecm=(m_j)\in\Z^r \\ |\vecm|=-q}}
              a_{\vecl,\vecm}^r(q)
                \left(  \prod_{j=1}^r(k_j)_{l_j}  \right)
                  \Li_{\veck+\vecm}(t).
  \end{equation}  
}
\medskip

See Proposition \ref{prop:deszeta can be written as sum of mzfs} for definition of the Pochhammer symbol $(k_j)_{l_j}$, and see \eqref{eqn:coeff. of q-th differential of mathcalG} for the symbol $a_{\vecl,\vecm}^r(q)$.


The construction of this paper goes as follows:
In \S\ref{sec: The MZFs and desingularized MZFs}, 
we recall the MZF and its desingularization.
And then, We observe that the shuffle product formula that holds for MZVs holds for 
special values of the desingularized MZFs with depth $1$ at positive integer points.
In \S\ref{sec: Construstion of desLi(veck)(t)}, 
we review a certain Hopf algebra derived from the differential relation of multiple polylogarithms.
We show that special values of $\deszeta_r(s_1,\dots,s_r)$ at all-negative integer points have an
``{\it iterated differential expression}''.
After that, we introduce a certain function $\desLi(\veck)(t)$ defined by an iterated integral expression.
In \S\ref{sec: Linear combinations of MPLs}, 
we define a function $Z(\veck;t)$ to express $\desLi(\veck)(t)$ as a finite linear combination of multiple polylogarithms. 
At the end of this section, we show that $Z(\veck;t)$ and $\desLi(\veck)(t)$ are equal.
In \S\ref{sec: Main results}, 
we introduce an $(r+1)$-fold analytic function $\hatdeszeta_r(s_1,\dots,s_r)(t)$ defined by a Hankel contour integral. 
By showing $\hatdeszeta_r(\veck)(t)$ and $Z(\veck;t)$ are equal, we prove our main theorem.
In Appendix \ref{sec:Explicit formula of Mq(veck)}, we show explicit expressions of $Z(\veck;t)$ in terms of $\Li_{\veck}(t)$.\\[3mm]
\noindent
{\it Acknowledgments.}
We would like to thank Hidekazu Furusho and Kohji Matsumoto for their helpful comments.

\section{The MZFs and desingularized MZFs}
 \label{sec: The MZFs and desingularized MZFs}

In this section, we review the multiple zeta functions in \S\ref{subsec: The multiple zeta functions}.
In \S\ref{Desingularized MZFs}, 
we recall the definition of the desingularized MZFs which is introduced by 
Furusho, Komori, Matsumoto, and Tsumura in \cite{FKMT1}, 
and explain some properties.
In \S\ref{sec:Observation on the shuffle product of deszeta1(k)}, 
we observe that special values of the desingularized MZFs in one variable case at all positive integers 
satisfy the shuffle product formula.

\subsection{The MZFs}
 \label{subsec: The multiple zeta functions}
In this subsection, we recall the multiple zeta functions.

The Euler-Zagier {\it multiple zeta-function} (MZF for short) is the $r$-fold complex analytic function 
defined by
\[
  \zeta_r(s_1,s_2,\dots,s_r) 
   = \sum_{0<n_1<n_2<\cdots<n_r}
      \frac{1}
           {n_1^{s_1}n_2^{s_2}\cdots n_r^{s_r}}.
\]
It converges absolutely in the region
\[
  \mathcal{D}:= \{ (s_1,\dots,s_r)\in\C^r  \,|\, 1\le j \le r,\, \re(s_{r-j+1}+\cdots+s_r)>j \}.
\]
In the early 2000s, Zhao (\cite{Zhao}) and Akiyama, Egami, and Tanigawa (\cite{AET}) 
independently showed that the MZF can be meromorphically continued to $\C^r$.
In particular, in \cite{AET}, the set of singularities of the MZF is determined as follows;
\begin{equation*}
 \begin{split}
   s_r&=1, \\
   s_{r-1} + s_r &= 2, 1, 0, -2, -4,-6,\dots, \\
   \sum_{i=1}^{k}s_{r-i+1} &\in \Z_{\le k}, \quad (k=3,4,\dots,r).
 \end{split}
\end{equation*}
This shows that almost all non-positive integer points are located in the above singularities. 
In addition, it is known that they are points of indeterminacy.
Only the special values $\zeta(-k)$ ($k\in\N_0$) and
 $\zeta_2(-k_1,-k_2)$ ($k_1,k_2\in\N_0$ with $k_1+k_2$ odd) are well-defined. 
 It is one of the fundamental problems to give a nice meaning of ``$\zeta_r(-k_1,\dots,-k_r)$'' for $k_1,\dots,k_r\in\N_0$.
Regarding this, desingularized method (\S\ref{Desingularized MZFs}) and the renormalization method (\S\ref{subsec:a certain Hopf algebra}, \ref{subsec:Generalization of the shuffle type renormalization}) are known.

\subsection{Desingularized MZFs}
 \label{Desingularized MZFs}
 In this subsection, we recall the definition of the desingularized MZFs and their remarkable properties.

We start with the generating function 
\footnote{It is denoted by $\tilde{\mathfrak{H}}_r((t_j);(1);c)$ in \cite{FKMT1}.}
$\tilde{\mathfrak{H}}_r(t_1,\dots,t_r;c)\in\C[[t_1,\dots,t_r]]$ 
(cf. \cite[Definition 1.9]{FKMT1}); for $c\in\R$, we put
\begin{equation*}
 \begin{split}
  \tilde{\mathfrak{H}}_r(t_1,\dots,t_r;c)
   &:= \prod_{j=1}^{r}
            \left(
               \frac{1}{\exp(\sum_{k=j}^{r}t_k)-1}
                - \frac{c}{\exp(c\sum_{k=j}^{r}t_k)-1}
            \right) \\
   &= \prod_{j=1}^{r}
           \left(
               \sum_{m\ge1}
                (1-c^m)B_m
                 \frac{ (\sum_{k=j}^{r}t_k)^{m-1} }
                          {m!}
            \right).
 \end{split}
\end{equation*}
Here, $B_m$ ($m\in\N_0$) is the $m$-th Seki-Bernoulli number which is defined by
\[
    \frac{x}{e^x-1} := \sum_{m\ge0}
                                           \frac{B_m}{m!} x^m.
\]
We note that $B_0=1$, $B_1=-\frac{1}{2}$, $B_2=\frac{1}{6}$.

\begin{defn}[{\cite[Definition 3.1]{FKMT1}}]\label{def:definition of des mzf}
  For $s_1,\dots,s_r\in\C\setminus\Z$, 
  we define
  \begin{equation}\label{eqn:definition of des mzf}
      \deszeta_r(s_1,\dots,s_r) 
       := \lim_{ \substack{c \rightarrow 1 \\ c\ne 1} }
               \frac{1}{(1-c)^r}
                C(\vecs)
                  \int_{\mcalc^r}
                   \tilde{\mathfrak{H}}_r(t_1,\dots,t_r;c)
                    \prod_{k=1}^{r}t_k^{s_{k}-1}dt_k,
  \end{equation}
  where $\mcalc$ is the Hankel contour, that is, the path consisting of the positive axis (top side), 
  a circle around the origin of radius $\epsilon$ (sufficiently small), 
  and the positive real axis (bottom side).
  We put 
  \begin{equation}\label{eqn: coefficient of contour integration}
    C(\vecs):=C(s_1,\dots,s_k):=\prod_{k=1}^{r}\frac{1}{(e^{2\pi i s_k}-1)\Gamma(s_k)}.
  \end{equation}
  We call $\deszeta_r(s_1,\dots,s_r)$ the {\it desingularized MZF}.
\end{defn}

The next proposition guarantees the convergence of the right-hand side of 
\eqref{eqn:definition of des mzf}.

\begin{prop}[{\cite[Theorem 3.4]{FKMT1}}]\label{prop:holomolcity and analytic continuation of des mzf}
  The function $\deszeta_r(s_1,\dots,s_r)$ can be analytically continued to $\C^r$
  as an entire function in $(s_1,\dots,s_r)\in\C^r$ by the following integral expression;
  \begin{equation*}
     \deszeta_r(s_1,\dots,s_r)
      = C(\vecs) 
      \cdot
         \int_{\mcalc^r}
          \prod_{j=1}^{r}
           \left(\frac{1}{\exp(\sum_{k=j}^{r}t_k)-1}
                       - \frac{\sum_{k=j}^{r}t_k\exp(\sum_{k=j}^{r}t_k)}{(\exp(\sum_{k=j}^{r}t_k)-1)^2}
              \right) 
               \prod_{k=1}^{r}t_k^{s_{k}-1}dt_k.
  \end{equation*} 
\end{prop}

We review useful notations to show another wonderful property of $\deszeta_r(s_1,\dots,s_r)$.
For indeterminants $u_j$ and $v_j$ ($1\le j \le r$), we set
\begin{equation}\label{eqn:def of mathcalGr}
    \mathcal G_r:= \mcalg_r(u_1,\dots,u_r;v_1,\dots,v_r)
      := \prod_{j=1}^{r}
            \{ 1- (u_jv_j+\cdots+u_rv_r)(v_j^{-1}-v_{j-1}^{-1}) \}
\end{equation}
with the convention $v_0^{-1}:=0$, and we define the set of integers $\{\almr\}$ by 
  \begin{equation}\label{eqn:the seq of coeff of deszeta}
     \mcalg_r(u_1,\dots,u_r;v_1,\dots,v_r)
      = \sum_{ \substack{\vecl=(l_j)\in\N_0^r \\ \vecm(m_j)\in\Z^r \\ |\vecm|=0} }
           \almr 
            \prod_{j=1}^{r}
             u_{j}^{l_j}v_{j}^{m_j}.
  \end{equation}
Here, we put $|\vecm|:=m_1+\cdots+m_r$.

One of the remarkable properties of the desingularized MZFs is that it can be represented 
as a finite ``linear'' combination of MZFs.

\begin{prop}[{\cite[Theorem 3.8]{FKMT1}}]\label{prop:deszeta can be written as sum of mzfs}
  For $s_1,\dots,s_r\in\C$, we have the following equality 
  between meromorphic functions of the complex variables $(s_1,\dots,s_r)$;
  \[
    \deszeta_r(s_1,\dots,s_r) 
     = \sum_{ \substack{\vecl=(l_j)\in\N_0^r \\ \vecm(m_j)\in\Z^r \\ |\vecm|=0} }
          \almr 
           \left( \prod_{j=1}^{r} (s_j)_{l_j}\right)
            \zeta_r(s_1+m_1,\dots,s_r+m_r),
  \]
  where $(s)_k$ is the Pochhammer symbol, that is,
  $(s)_0:=1$ and $(s)_k:=s(s+1)\cdots(s+k-1)$ for $k\in\N$ and $s\in\C$.
\end{prop}

Let us see some examples of Proposition \ref{prop:deszeta can be written as sum of mzfs}. 

\begin{exa}\label{exa:example of deszeta_1 and deszeta_2}
  (\cite[\S4]{FKMT1}).
  In the case $r=1$, we have
  \begin{equation*}\label{eqn:deszeta_1 is  (1-s)zeta_1}
    \deszeta_1(s)=(1-s)\zeta(s).
  \end{equation*}
  We also see that $\deszeta_1(1)=-1$.

  In the case $r=2$, we have
  \begin{equation}\label{eqn:deszeta_2 can be written as a summation of zeta_2}
   \begin{split}
     \deszeta_2(s_1,s_2)
    = &(s_1-1)(s_2-1)\zeta_2(s_1,s_2)
         + s_2(s_2+1-s_1)\zeta_2(s_1-1,s_2+1) \\
       &- s_2(s_2+1)\zeta_2(s_1-2,s_2+2).
   \end{split}
  \end{equation}
  
  By \cite[Proposition 4.9]{FKMT2}, we have
  \begin{equation}
   \label{eqn:the special value deszeta_2(1,1)}
    \deszeta_2(1,1) = \frac{1}{2}.
  \end{equation}
\end{exa}

We consider the following generating function $Z(t_1,\dots,t_r)$ of 
the special value $\deszeta_r(-k_1,\dots,-k_r)$ ($k_1,\dots,k_r\in\N_0$) which will be employed in the later section:
\[
   Z(t_1,\dots,t_r)
    := \sum_{k_1,\dots,k_r \ge 0}
          \frac{ (-t_1)^{k_1}\cdots(-t_r)^{k_r} }
                   {k_1!\cdots k_r!}
            \deszeta_r(-k_1,\dots,-k_r).
\]
This is explicitly calculated as follows.

\begin{prop}[{\cite[Theorem 3.7]{FKMT1}}]\label{prop:explicit formula of gen. fun Z}
  We have
  \[
      Z(t_1,\dots,t_r) 
      = \prod_{i=1}^r
           \frac{(1-t_i-\cdots-t_r)e^{t_i+\cdots+t_r}-1}{(e^{t_i+\cdots+t_r}-1)^2}.
  \]
  In terms of $\deszeta_r(-k_1,\dots,-k_r)$ for $k_1,\dots,k_r\in\N_0$, 
  the above equation is reformulated to 
  \[
      \deszeta_r(-k_1,\dots,-k_r)
       = (-1)^{k_1+\cdots+k_r}
           \sum_{ 
            \substack{
                                 \nu_{1i}+\cdots+\nu_{ii}=k_{i} \\
                                 1\le i \le r
                               }
                       }
             \prod_{i=1}^r
              \frac{k_i!}{\prod_{j=i}^r\nu_{ij}!}
               B_{\nu_{ii}+\cdots+\nu_{ir}+1}.
  \]
\end{prop}

We recall that the first-named author investigated the following shuffle-type formula of special values of the desingularized MZFs.

\begin{prop}[{\cite[Proposition 4.8]{Komi2}}]\label{prop:recurrence relation of deszeta}
  For $s_1,\dots,s_{r-1}\in\C$ and $k\in\N_0$, we have
  \begin{equation*}
    \deszeta_r(s_1,\dots,s_{r-1},-k) 
      =  \sum_{i=0}^k
             \binom{k}{i}
               \deszeta_{r-1}(s_1,\dots,s_{r-1}-k+i)
                 \deszeta_1(-i).
  \end{equation*}
\end{prop}

\begin{prop}[{\cite[Theorem 2.7]{Komi3}}]\label{prop:shuffle type formula of deszeta}
  For $s_1,\dots,s_{p}\in\C$, $l_1,\dots,l_q\in\N_0$, we have
  \begin{equation*}\label{eqn:recurrence relation of deszeta}
   \begin{split}
    &\deszeta_{p}(s_1,\dots,s_{p})\deszeta_q(-l_1,\dots,-l_q) \\
     &=\sum_{
        \substack{i_b+j_b=l_b   \\   i_b,j_b\ge0   \\   1\le b \le q}
                 }
         \prod_{a=1}^q
          (-1)^{i_a}      
           \binom{l_a}{i_a}
           \deszeta_{p+q}(s_1,\dots,s_{p-1},s_p-i_1-\cdots-i_q,-j_1,\dots,-j_q).
    \end{split}
  \end{equation*}
\end{prop}

We see that the ``shuffle-type'' formula holds for products of special values at any integer points
and the ones at all-non-positive integer points.


\subsection{Observation on the shuffle product of $\deszeta_1(k)$}
\label{sec:Observation on the shuffle product of deszeta1(k)}
In this subsection, we present a simple, but important observation.

At first, we remark on some special values of the double zeta function. 
\begin{lem}[{\cite[Proposition 4]{AK}}]
 \label{lem:the limit value of (s-1)zeta(n,s) s goes to 1}
 For $n\in\N_{\ge2}$  we have
 \begin{equation}
  \label{eqn:the limit value of (s-1)zeta(n,s) s goes to 1}
   \lim_{s \rightarrow 1}(s-1)\zeta_2(n,s) = \zeta(n).
 \end{equation}
\end{lem}

\begin{proof}
It is a direct consequence of the expansion
\[
  \zeta_2(n,s) = \frac{\zeta(n)}{s-1} + O(1),
\]
where $O$ is Landau's notation.
\end{proof}

It is known the following for the case $n\le0$.

  \begin{rem}[{\cite[Example 4.2]{FKMT2}}]\label{rem:special values of MZFs for r=2}
  The following two hold:
  \begin{align*}
     &\zeta_2(0,s)= \zeta(s-1)-\zeta(s), \\ 
     &\zeta_2(-1,s)= \frac{1}{2}\left\{ \zeta(s-2)-\zeta(s-1) \right\}
  \end{align*}
  for $s\in\C$ except for singularities.
  Hence, we have
  \begin{equation}\label{eqn:zeta(-1,s)=1/2 zeta(0,s-1)}
  \zeta_2(-1,s)=\frac{1}{2}\zeta_2(0,s-1)
  \end{equation}
  for $s,s-1$ except for singularities.
  \end{rem} 

By the elementary calculation, we find the following.

\begin{prop}\label{prop:shuffle relation of dep1 des MZF}
   For $n$, $m\in\N$, we have 
  \begin{equation}
   \label{eqn:shuffle rel. of zeta^des1}
    \zeta^{\des}_1(n) \zeta_1^{\des}(m) 
     = \sum_{j=1}^{n+m-1}
          \left\{\binom{j-1}{n-1}+\binom{j-1}{m-1}\right\}
           \zeta^{\des}_2(n+m-j,j), 
  \end{equation}
  where we use the usual convention $\binom{k}{i}=0$ for $k<i$.
\end{prop}

\begin{proof}
  If $n=m=1$, by \eqref{eqn:the special value deszeta_2(1,1)}, we have 
  \[
     \deszeta_1(1)^2 = 1 = 2\deszeta_2(1,1).
  \]
  
 We prove \eqref{eqn:shuffle rel. of zeta^des1} for $m,n\geq2$.
  To save space, we set 
  \[
    b_{j,n}:=\binom{j-1}{n-1},\  (j,\,n\in\N).
  \]
  
  Since $n$ and $m$ are symmetric, we compute the following: 
  \begin{align*}
    &\sum_{j=1}^{n+m-1}
       \binom{j-1}{n-1}
        \zeta^{\des}_2(n+m-j,j) \\
    &= \sum_{j=1}^{n+m-1}
           b_{j,n}
            \Big\{ (n+m-j-1)(j-1)\zeta_2(n+m-j,j) 
            \nonumber\\
    & \quad + j(2j+1-n-m)\zeta_2(n+m-j-1,j+1) 
           - j(j+1)\zeta_2(n+m-j-2,j+2)
        \Big\}. \nonumber\\
   \intertext{Notice that the equality holds by \eqref{eqn:deszeta_2 can be written as a summation of zeta_2}. By rearranging each term, we have}
    &= \sum_{j=3}^{n+m-1}
        \Big\{ b_{j,n}(n+m-j-1)(j-1)
               + b_{j-1,n}(j-1)(2j-1-n-m) \nonumber\\
    &\hspace{17mm}
               - b_{j-2,n}(j-2)(j-1)
        \Big\}\zeta_2(n+m-j,j) + R^1_{n,m} + R^2_{n,m}, \nonumber
  \end{align*}
  where
  {\small
  \begin{equation*}
   \begin{split}
    R^1_{n,m} &:=\ b_{1,n}(n+m-2)\zeta(n+m-1) 
                         - b_{2,n}(n+m-3)\zeta_2(n+m-2,2) \\
                        &\hspace{3mm} - b_{1,n}(3-n-m)\zeta_2(n+m-2,2), \\
    R^2_{n,m} &:= \ b_{n+m-1,n}(n+m-1)^2\zeta_2(0,n+m) 
                              - b_{n+m-2,n}(n+m-2)(n+m-1)\zeta_2(0,n+m) \\
                          &\hspace{3mm} - b_{n+m-1,n}(n+m-1)(n+m)\zeta_2(-1,n+m+1).
   \end{split}          
  \end{equation*}}
  We use \eqref{eqn:the limit value of (s-1)zeta(n,s) s goes to 1} in the definition of $R_{n,m}^1$. 
  By using $\binom{n-1}{r}=\binom{n}{r}\frac{n-r}{n}$ (i.e. $b_{j-1,n}=b_{j,n}\frac{j-n}{j-1}$),
  we have
  {\small
  \begin{align}\label{eqn:calculation of the coefficients in the RHS of shuffle rel. of zeta^des1}
     &b_{j,n}(n+m-j-1)(j-1) 
      + b_{j-1,n}(j-1)(2j-1-n-m) 
       - b_{j-2,n}(j-2)(j-1) \\
     &= \Big\{
                   (n+m-j-1)(j-1) + (j-n)(2j-1-n-m) - (j-n-1)(j-n)
           \Big\} b_{j,n} \nonumber\\
     &= (1-n)(1-m)b_{j,n}. \nonumber
  \end{align}}
  
  We next compute $R^1_{n,m}$ and $R_{n,m}^2$. 
  We can easily see that the following holds for $n,m\in\N_{\ge2}$:
  \begin{equation}
   \label{eqn:cal of the R^1_n}
    R^1_{n,m} = \delta_{n,2}(m-1)\zeta_2(n+m-2,2)=b_{2,n}(m-1)\zeta_2(n+m-2,2),
  \end{equation}
  where $\delta_{i,j}$ is the Kronecker delta.
  By using $b_{j-1,n}=b_{j,n}\frac{j-n}{j-1}$, we have
  \begin{align}
   \label{eqn:cal.1 of the R^2_n}
     R^2_{n,m} &= b_{n+m-1,n}(n+m-1)^2\zeta_2(0,n+m) \nonumber\\
                     &\hspace{3.5mm} - b_{n+m-1,n}(m-1)(n+m-1)\zeta_2(0,n+m) \nonumber\\
                     &\hspace{3.5mm} - b_{n+m-1,n}(n+m-1)(n+m)\zeta_2(-1,n+m+1). \nonumber
  \intertext{By using \eqref{eqn:zeta(-1,s)=1/2 zeta(0,s-1)}, we obtain}
     R^2_{n,m}&= b_{n+m-1,n}(n-m) (n+m-1) \zeta_2(-1,n+m+1).
  \end{align}
  By \eqref{eqn:calculation of the coefficients in the RHS of shuffle rel. of zeta^des1}, 
  \eqref{eqn:cal of the R^1_n}, and
  \eqref{eqn:cal.1 of the R^2_n}, we have
  \begin{align*}
    &\sum_{j=1}^{n+m-1}
       \binom{j-1}{n-1}
        \zeta^{\des}_2(n+m-j,j) \\
     &= (1-n)(1-m)
            \sum_{j=3}^{n+m-1}
              b_{j,n}
               \zeta_2(n+m-j,j) 
         + b_{2,n}(m-1)\zeta_2(n+m-2,2) \\
     &\quad + b_{n+m-1,n}(n-m) (n+m-1) \zeta_2(-1,n+m+1).
     \intertext{Because $b_{1,n}=0$ for $n\geq2$, we get}
     &= (1-n)(1-m)
            \sum_{j=1}^{n+m-1}
              b_{j,n}
               \zeta_2(n+m-j,j) \\
     &\quad + b_{n+m-1,n}(n-m) (n+m-1) \zeta_2(-1,n+m+1).
   \end{align*} 
  Since $\binom{p+q}{p}=\binom{p+q}{q}$ for $p$, $q\in\N$, we have $b_{n+m-1,n} = b_{n+m-1,m}$.
  Thus, we have 
  \begin{align*}
  &\sum_{j=1}^{n+m-1}
          \left\{\binom{j-1}{n-1}+\binom{j-1}{m-1}\right\}
           \zeta^{\des}_2(n+m-j,j) \\
  &=(1-n)(1-m) \sum_{j=1}^{n+m-1} (b_{j,n} + b_{j,m}) \zeta_2(n+m-j,j).
  \intertext{Because $m,n\geq2$, by the shuffle product formula of MZVs, we calculate}
  &= (1-n)(1-m)\zeta(n)\zeta(m)
  =\zeta_1^{\des}(n) \zeta_1^{\des}(m).
  \end{align*}
  Hence, the equation \eqref{eqn:shuffle rel. of zeta^des1} holds for $m,n\geq2$.
  In the same way, we can prove \eqref{eqn:shuffle rel. of zeta^des1} for $n=1$ and $m\geq2$ (or $n\geq2$ and $m=1$).
\end{proof}

In Theorem \ref{cor:shuffle product of deszeta} in \S\ref{sec: Main results}, 
we will generalize this proposition to higher depth.


\section{Construction of $\desLi(\veck)(t)$}
\label{sec: Construstion of desLi(veck)(t)}
In this section, we review the Hopf algebra $(\mathcal H,\shuffle_0,\Delta_0)$ and its Hopf subalgebra $(\mathcal H_-,\shuffle_0,\Delta_0)$ which are introduced in \cite{EMS1}, and we prove some algebraic relations on $\mathcal H_-$ (Proposition \ref{prop:equality of sum of shuffle0} and Proposition \ref{prop:recurrence formula for shuffle0}).
By using these relations, in \S \ref{subsec:Generalization of the shuffle type renormalization}, we give a reinterpretation of the desingularized MZFs at non-positive integers by the renormalization.
In \S\ref{sec:Definition of hypLiveck(t)}, we give a quick review of multiple polylogarithms and
 introduce functions $\desLi(\veck)(t)$. And then, we mention the relationship between these functions for $\veck\in\Z_{\leq0}^r$ and special values $\deszeta_r(\veck)$. 

\subsection{A certain Hopf algebra}\label{subsec:a certain Hopf algebra}
We recall the definition of the Hopf algebra introduced in \cite{EMS1} and its properties.
Put $\Q\langle d,j,y \rangle$ to be the non-commutative polynomial ring generated by three letters $d,j,y$ with $dj=jd=1$.
We define the product $\shuffle_0:\Q\langle d,j,y \rangle^{\otimes2} \rightarrow \Q\langle d,j,y \rangle$ by $w\shuffle_01:=1\shuffle_0w:=w$ and
\begin{align}\label{eqn:def of shuffle0}
yu\shuffle_0v &:= u\shuffle_0yv := y(u\shuffle_0v), \nonumber\\
ju\shuffle_0jv &:= j(u\shuffle_0jv) + j(ju\shuffle_0v), \nonumber\\
du\shuffle_0dv &:= d(u\shuffle_0dv) - u\shuffle_0d^2v, \\
du\shuffle_0jv &:= d(u\shuffle_0jv) - u\shuffle_0v, \nonumber\\
ju\shuffle_0dv &:= d(ju\shuffle_0v) - u\shuffle_0v, \nonumber
\end{align}
for any words $u,v,w$ in $\Q\langle d,j,y \rangle$.
Then the pair $(\Q\langle d,j,y \rangle, \shuffle_0)$ forms a non-commutative, non-associative $\Q$-algebra.
We define $\mathcal T$ to be the linear subspace over $\Q$ of $\Q\langle d,j,y \rangle$ generated by
\footnote{Note that the symbol $j^{-k}$ means $d^k$ for $k\in\N$.}
$$
\left\{ j^{k_1}y\cdots j^{k_{r-1}}yj^{k_r}\in \Q\langle d,j,y \rangle\ |\ r\in\N, k_1,\dots,k_r\in\Z, k_r\neq0 \right\},
$$
and define $\mathcal L$ to be the two-sided ideal of $(\Q\langle d,j,y \rangle, \shuffle_0)$ generated by
\footnote{That is, the words $u$ and $v$ are represented as $j^{k_1}y\cdots j^{k_r}y$ for $r\in\N$ and $k_1,\dots,k_r\in\Z$.}
\begin{equation}\label{def: definition of two sided ideal mathcalL}
\left\{ j^k\{d(u\shuffle_0v)-du\shuffle_0v-u\shuffle_0dv\}\ |\ k\in\Z, \mbox{$u,v$:words of $\Q\langle d,j,y \rangle$ ending in $y$} \right\}.
\end{equation}
It is proved that the subspace $\mathcal T$ is a two-sided ideal of $(\Q\langle d,j,y \rangle, \shuffle_0)$ in \cite[Lemma 3.4]{EMS1}.
We define the quotient algebra
\begin{equation}\label{def: Hopf alg. mathcalH}
\mathcal H:=\Q\langle d,j,y \rangle/(\mathcal T + \mathcal L).
\end{equation}
Then the pair $(\mathcal H, \shuffle_0)$ is a commutative, associative $\Q$-algebra (\cite[Proposition 3.5]{EMS1}).

We consider the non-commutative polynomial ring $\Q\langle d,y \rangle$ generated by two letters $d,y$.
The pair $(\Q\langle d,y \rangle, \shuffle_0)$ forms a subalgebra of $(\Q\langle d,j,y \rangle,\shuffle_0)$.
We put
$$
\mathcal T_-:=\Q\langle d,y \rangle \cap \mathcal T, \qquad \mathcal L_-:=\Q\langle d,y \rangle \cap \mathcal L.
$$
and define
$$
\mathcal H_-:=\Q\langle d,y \rangle/(\mathcal T_- + \mathcal L_-).
$$
Then there exists a coproduct $\Delta_0$ such that the triple $(\mathcal H_-, \shuffle_0,\Delta_0)$ forms a commutative Hopf algebra over $\Q$.
In \cite[Proposition 2.5]{Komi}, the reduced coproduct\footnote{The reduced coproduct $\widetilde{\Delta}_0$ is defined by $\widetilde{\Delta}_0(w):=\Delta_0(w)-1\otimes w-w\otimes 1$ for any words w in $\mathcal H_-$.} $\widetilde{\Delta}_0$ is explicitly given by
\begin{align}\label{eqn:explicit formula reduced coproduct}
&\widetilde{\Delta}_0(d^{k_1}y\cdots d^{k_r}y)
=\sum_{i+j=k_1}\binom{k_1}{i}d^{i}y\otimessym d^{j}d^{k_2}y\cdots d^{k_r}y \\
&+\sum_{p=2}^{r-1}\sum_{\substack{i_l+j_l=k_l \\ 1\leq l\leq p}}\prod_{a=1}^p\binom{k_a}{i_a}
\sum_{\substack{\{u_q,v_q\}=\{d^{i_q},d^{j_q}y\} \\ 1\leq q\leq p-1}}
u_1\cdots u_{p-1}d^{i_p}y\otimessym v_1\cdots v_{p-1}d^{j_p}d^{k_{p+1}}y\cdots d^{k_r}y, \nonumber
\end{align}
for $r\geq2$ and $k_1,\dots,k_r\in\N_0$.
Here, $w_1\otimessym w_2:=w_1\otimes w_2 + w_2\otimes w_1$ for $w_1,w_2\in\mathcal H_-$.
\begin{prop}[{\cite[Corollary 3.24 for $\lambda=0$]{EMS1}}]\label{prop:EMS:cor3.24}
For any word $w(\neq1)$ in $\mathcal H_-$, we have
$$
(2^{\dep(w)}-2)w=\shuffle_0\circ \widetilde{\Delta}_0(w).
$$
Here, the symbol $\dep(w)$ means the number of $y$ appearing in the word $w$.
\end{prop}
We prove the following by using the above proposition.
\begin{prop}\label{prop:equality of sum of shuffle0}
Let $r\geq3$ and $k_1,\dots,k_r\in\N_0$.
We have
\begin{align*}
&\sum_{\substack{i_l+j_l=k_l \\ 1\leq l\leq p}}\prod_{a=1}^p\binom{k_a}{i_a}
u_1\cdots u_{p-1}d^{i_p}y\shuffle_0 v_1\cdots v_{p-1}d^{j_p}d^{k_{p+1}}y\cdots d^{k_r}y \\
&\hspace{3cm}=\sum_{i+j=k_1}\binom{k_1}{i}d^{i}y\shuffle_0 d^{j}d^{k_2}y\cdots d^{k_r}y,
\end{align*}
for $p\in\{2,\dots,r-1\}$ and $(u_q,v_q)\in\{(d^{i_q},d^{j_q}y),(d^{j_q}y,d^{i_q})\}$ $(1\leq q\leq p-1)$.
\end{prop}
\begin{proof}
For $r\geq1$, we consider the following generating functions
$$
\mathcal G(t_1,\dots,t_r)
:=\sum_{k_1,\dots,k_r\geq0}
\frac{t_1^{k_1}\cdots t_r^{k_r}}{k_1!\cdots k_r!}d^{k_1}y\cdots d^{k_r}y
 \ \left( \in \mathcal H_-[[t_1,\dots,t_r]] \right).
$$
When $r=2$, by using the equation \eqref{eqn:explicit formula reduced coproduct} and Proposition \ref{prop:EMS:cor3.24}, we have
$$
d^{k_1}yd^{k_2}y=\sum_{i+j=k_1}\binom{k_1}{i}d^{i}y\shuffle_0 d^{j+k_2}y,
$$
for $k_1,k_2\in\N_0$. 
So we get
\begin{align*}
\mathcal G(t_1,t_2)
&=\sum_{k_1,k_2\geq0}
\frac{t_1^{k_1} t_2^{k_2}}{k_1!k_2!}
\sum_{i+j=k_1}\binom{k_1}{i}d^{i}y\shuffle_0 d^{j+k_2}y
=\sum_{i,j,k_2\geq0}
\frac{t_1^{i+j} t_2^{k_2}}{i!j!k_2!}
d^{i}y\shuffle_0 d^{j+k_2}y \\
&=\left\{
\sum_{i\geq0}\frac{t_1^{i}}{i!}
d^{i}y
\right\}
\shuffle_0 
\left\{
\sum_{k\geq0}
\frac{1}{k!}
d^{k}y
\sum_{j+k_2=k}\binom{k}{j}t_1^{j} t_2^{k_2}
\right\} \\
&=\mathcal G(t_1)
\shuffle_0 
\left\{
\sum_{k\geq0}
\frac{(t_1+t_2)^k}{k!}
d^{k}y
\right\} 
=\mathcal G(t_1) \shuffle_0 \mathcal G(t_1+t_2).
\end{align*}
Therefore, we get
\begin{equation}\label{eqn:recurrence formula of G for r=2}
\mathcal G(t_1,t_2)=\mathcal G(t_1) \shuffle_0 \mathcal G(t_1+t_2).
\end{equation}
When $r\geq3$, by direct calculation, we know
\begin{align*}
&\sum_{k_1,\dots,k_r\geq0}
\frac{t_1^{k_1}\cdots t_r^{k_r}}{k_1!\cdots k_r!}
\left\{
\sum_{\substack{i_l+j_l=k_l \\ 1\leq l\leq p}}\prod_{a=1}^p\binom{k_a}{i_a}
u_1\cdots u_{p-1}d^{i_p}y\shuffle_0 v_1\cdots v_{p-1}d^{j_p}d^{k_{p+1}}y\cdots d^{k_r}y 
\right\} \\
&=\mathcal G(t_1\circ_1t_2\circ_1\dots\circ_{p-1}t_p)\shuffle_0\mathcal G(t_1\diamond_1t_2\diamond_2\dots\diamond_{p-1}t_p+t_{p+1},\dots,t_r).
\end{align*}
Here, the symbols $\circ_q$ and $\diamond_q$ are defined by
$$
(\circ_q,\diamond_q)
:=\left\{\begin{array}{ll}
(+,\ \scalebox{2}{,}\ ) & \left( (u_q,v_q)=(d^{i_q}, d^{j_q}y) \right), \\
(\ \scalebox{2}{,}\ ,+) & \left( (u_q,v_q)=(d^{j_q}y, d^{i_q}) \right),
\end{array}
\right.
$$
for $1\leq q\leq p-1$.
By using \eqref{eqn:recurrence formula of G for r=2} and by induction on $r\geq3$, we get
\begin{align*}
&\mathcal G(t_1\circ_1t_2\circ_1\dots\circ_{p-1}t_p)\shuffle_0\mathcal G(t_1\diamond_1t_2\diamond_2\dots\diamond_{p-1}t_p+t_{p+1},\dots,t_r) \\
&=\mathcal G(t_1)\shuffle_0\mathcal G(t_1+t_2)\shuffle_0\cdots\shuffle_0\mathcal G(t_1+\cdots+t_r) \\
&=\mathcal G(t_1)\shuffle_0\mathcal G(t_1+t_2,t_3,\dots,t_r).
\end{align*}
Because we have
\begin{align*}
\mathcal G(t_1)\shuffle_0\mathcal G(t_1+t_2,t_3,\dots,t_r)
=\sum_{k_1,\dots,k_r\geq0}
\frac{t_1^{k_1}\cdots t_r^{k_r}}{k_1!\cdots k_r!}
\left\{
\sum_{i+j=k_1}\binom{k_1}{i}d^{i}y\shuffle_0 d^{j}d^{k_2}y\cdots d^{k_r}y
\right\},
\end{align*}
hence we obtain the claim.
\end{proof}

\begin{prop}\label{prop:recurrence formula for shuffle0}
For $r\geq2$ and $k_1,\dots,k_r\in\N_0$, we have
\begin{equation*}
d^{k_1}y\cdots d^{k_r}y=\sum_{i+j=k_1}\binom{k_1}{i}d^{i}y\shuffle_0 d^{j}d^{k_2}y\cdots d^{k_r}y.
\end{equation*}
\end{prop}
\begin{proof}
When $r=2$, we obtain the claim by the equation \eqref{eqn:explicit formula reduced coproduct}.
Set $r\geq3$.
By using the equation \eqref{eqn:explicit formula reduced coproduct} and Proposition \ref{prop:EMS:cor3.24}, we have\footnote{Note that $\shuffle_0(w_1\otimessym w_2)=2(w_1\shuffle_0w_2)$ for words $w_1,w_2$ in $\mathcal H_-$.}
\begin{align*}
d^{k_1}y\cdots d^{k_r}y
&=\frac{1}{2^{r-1}-1}\left\{
\sum_{i+j=k_1}\binom{k_1}{i}d^{i}y\shuffle_0 d^{j}d^{k_2}y\cdots d^{k_r}y
\right. \\
&\hspace{-2.1cm}\left.
+\sum_{p=2}^{r-1}\sum_{\substack{i_l+j_l=k_l \\ 1\leq l\leq p}}\prod_{a=1}^p\binom{k_a}{i_a}
\sum_{\substack{\{u_q,v_q\}=\{d^{i_q},d^{j_q}y\} \\ 1\leq q\leq p-1}}
u_1\cdots u_{p-1}d^{i_p}y\shuffle_0 v_1\cdots v_{p-1}d^{j_p}d^{k_{p+1}}y\cdots d^{k_r}y
\right\}.
\intertext{By using Proposition \ref{prop:equality of sum of shuffle0}, we get}
&=\frac{1}{2^{r-1}-1}\left\{
\sum_{i+j=k_1}\binom{k_1}{i}d^{i}y\shuffle_0 d^{j}d^{k_2}y\cdots d^{k_r}y
\right. \\
&\hspace{3cm}\left.
+(2^{r-1}-2)\sum_{i+j=k_1}\binom{k_1}{i}d^{i}y\shuffle_0 d^{j}d^{k_2}y\cdots d^{k_r}y
\right\} \\
&=\sum_{i+j=k_1}\binom{k_1}{i}d^{i}y\shuffle_0 d^{j}d^{k_2}y\cdots d^{k_r}y.
\end{align*}
Hence, we obtain the claim.
\end{proof}

\subsection{Generalization of the shuffle type renormalization}\label{subsec:Generalization of the shuffle type renormalization}
In this subsection, we first recall the algebraic Birkhoff decomposition which is a fundamental tool in the work of Connes and Kreimer (\cite{CK}) on their Hopf algebraic approach to the renormalization of the perturbative quantum field theory.

Let $\Q[[z,z^{-1}]$ be the $\Q$-algebra consisting of all formal Laurent series.
Put $G=G(\mathcal H_-,\Q[[z,z^{-1}])$ to be the set of all $\Q$-linear map $\varphi:\mathcal H_- \rightarrow \Q[[z,z^{-1}]$ with $\varphi(1)=1$.
For $\varphi,\psi\in G$, we define the \textit{convolution} $\varphi*\psi\in G$ by
$$
\varphi*\psi:=m \circ (\varphi \otimes \psi) \circ \Delta_0.
$$
Here, $m$ is the natural product of $\Q[[z,z^{-1}]$.
Then the pair $(G,*)$ forms a group, whose unit is given by a map $e=u \circ \varepsilon_{\mathcal H_-}$. Here, the map $u$ is the natural unit of $\Q[[z,z^{-1}]$ and the map $\varepsilon_{\mathcal H_-}$ is the counit of $\mathcal H_-$ (see \cite{Man} for detail).

\begin{thm}[{\cite{CK}, \cite{EMS1}}: \textbf{algebraic Birkhoff decomposition}]
\label{thm:algebraic Birkhoff decomposition}
For $\varphi\in G$, there exist unique $\Q$-linear maps $\varphi_+:\mathcal H_- \rightarrow \Q[[z]]$ and $\varphi_-:\mathcal H_- \rightarrow \Q[z^{-1}]$ with $\varphi_-(1)=1\in\Q$ such that
\footnote{Note that the symbol $\varphi_-^{*-1}$ means the inverse element of $\varphi_-$ in $(G(\mathcal H_-,\Q[[z,z^{-1}]),*)$.}
$$
\varphi=\varphi_-^{*-1}*\varphi_+.
$$
Moreover, the maps $\varphi_+$ and $\varphi_-$ form algebra homomorphisms if the map $\varphi$ is an algebra homomorphism.
\end{thm}
\begin{rem}\label{rem:recurrence formula of varphi+,-}
We consider the projection $\pi: \Q[[z,z^{-1}] \rightarrow \Q[z^{-1}]$ defined by
$$
\pi\left( \sum_{n\geq -N}a_nz^n \right):=\sum_{n= -N}^{-1}a_nz^n
$$
for $a_n\in\Q$ and $N\in\N_0$.
By the above theorem, we can inductively calculate $\varphi_+$ and $\varphi_-$ by
\begin{align*}
& \varphi_-(w)
= -\pi \left(
\varphi(w)
+ \sum_{(w)}\varphi_-(w')\varphi(w'')
\right), \\
& \varphi_+(w)
= (\Id-\pi) \left(
\varphi(w)
+ \sum_{(w)}\varphi_-(w')\varphi(w'')
\right).
\end{align*}
Here, we denote Sweedler's notation of the reduced coproduct $\widetilde{\Delta}_0$ by
$$
\widetilde{\Delta}_0(w)=\sum_{(w)}w'\otimes w''.
$$
Note that, when $\varphi(\mathcal H_-) \subset \Q[[z]]$, we have
\begin{equation}\label{eqn:recurrence formula of varphi+,- for formal power series}
\varphi_-(w)=0,\qquad
\varphi_+(w)= \varphi(w).
\end{equation}
\end{rem}

Let $f(z)$ be a Laurent series in $\Q[[z,z^{-1}]$.
We define the $\Q$-linear map $\phi:\mathcal H_- \rightarrow \Q[[z,z^{-1}]$ by $\phi(1):=1$ and
\begin{equation}\label{eqn:def of phi by iterated differential}
\phi(d^{k_1}y\cdots d^{k_r}y)
:=\partial_z^{k_1}\Bigl[ f(z)
\partial_z^{k_2}\left[ f(z)
\cdots \partial_z^{k_r}[f(z)] \cdots
\right]
\Bigr],
\end{equation}
for $k_1,\dots,k_r\in\N_0$.
Here, the symbol $\partial_z$ is the derivative by $z$.
By the Leibniz rule of the derivative $\partial_z$, we know that the map $\phi$ forms an algebra homomorphism.
By applying the above theorem to this map $\phi$, we get the algebra homomorphism $\phi_+:\mathcal H_- \rightarrow \Q[[z]]$.
We define $F(k_1,\dots,k_r)\in\Q$ by
$$
F(k_1,\dots,k_r):=\lim_{z \rightarrow 0} \phi_+(d^{k_r}y\cdots d^{k_1}y)
$$
for $r\in\N$ and $k_1,\dots,k_r\in\N_0$.
We set the generating function of $F(k_1,\dots,k_r)$ by
$$
\mathcal F(z_1,\dots,z_r)
:=\sum_{k_1,\dots,k_r\geq0}
\frac{z_1^{k_1}\cdots z_r^{k_r}}{k_1!\cdots k_r!}F(k_1,\dots,k_r).
$$
The following explicit formulae of $\mathcal F(z_1,\dots,z_r)$ holds.
\begin{thm}\label{thm:explicit formula of F}
For $r\in\N$, we have
$$
\mathcal F(z_1,\dots,z_r)
=f_{\geq0}(z_r)f_{\geq0}(z_{r-1}+z_r)\cdots f_{\geq0}(z_1+\cdots+z_r).
$$
Here, $f_{\geq0}(z)$ is defined by
$$
f_{\geq0}(z):=\sum_{n_\geq0}\frac{a_n}{n!} z^n
$$
for $f(z)=\sum_{n\geq -N}\frac{a_n}{n!}z^n$ $(N\in\N_0)$ with $a_n\in\Q$.
\end{thm}
\begin{proof}
When $r=1$, we have $\widetilde{\Delta}_0(d^ky)=0$, so by Remark \ref{rem:recurrence formula of varphi+,-}, we calculate
\begin{align*}
F(k_1)
=\lim_{z \rightarrow 0} \phi_+(d^{k_1}y)
=\lim_{z \rightarrow 0} (\Id-\pi) \left( \phi(d^{k_1}y) \right)
=\lim_{z \rightarrow 0} (\Id-\pi) \left( \partial_z^{k_1}[f(z)] \right)
=a_{k_1}.
\end{align*}
Therefore, we get
\begin{equation}\label{eqn:r=1 for gen. fun. of F}
\mathcal F(z_1)
=\sum_{k_1\geq0}
\frac{z_1^{k_1}}{k_1!}F(k_1)
=\sum_{k_1\geq0}
\frac{z_1^{k_1}}{k_1!}a_{k_1}
=f_{\geq0}(z_1).
\end{equation}
Let $r\geq2$.
By Proposition \ref{prop:recurrence formula for shuffle0}, we have
\begin{align*}
F(k_1,\dots,k_r)
&=\lim_{z \rightarrow 0} \phi_+(d^{k_r}y\cdots d^{k_1}y)
=\lim_{z \rightarrow 0} \phi_+\left(
\sum_{i+j=k_r}\binom{k_r}{i}d^{i}y\shuffle_0 d^{j}d^{k_{r-1}}y\cdots d^{k_1}y
\right) \\
&=\lim_{z \rightarrow 0} 
\sum_{i+j=k_r}\binom{k_r}{i}
\phi_+\left( d^{i}y \right)
\phi_+\left( d^{j}d^{k_{r-1}}y\cdots d^{k_1}y \right) \\
&=\sum_{i+j=k_r}\binom{k_r}{i}
F(i)F(k_1,\dots,k_{r-2},k_{r-1}+j).
\end{align*}
So we calculate
\begin{align*}
&\mathcal F(z_1,\dots,z_r)
=\sum_{k_1,\dots,k_r\geq0}
\frac{z_1^{k_1}\cdots z_r^{k_r}}{k_1!\cdots k_r!}F(k_1,\dots,k_r) \\
&=\sum_{k_1,\dots,k_r\geq0}
\frac{z_1^{k_1}\cdots z_r^{k_r}}{k_1!\cdots k_r!}
\sum_{i+j=k_r}\binom{k_r}{i}
F(i)F(k_1,\dots,k_{r-2},k_{r-1}+j) \\
&=\left\{
\sum_{i\geq0}\frac{z_r^{i}}{i!}F(i)
\right\}\left\{
\sum_{k_1,\dots,k_{r-2},k\geq0}
\frac{z_1^{k_1}\cdots z_{r-2}^{k_{r-2}}}{k_1!\cdots k_{r-2}!k!}F(k_1,\dots,k_{r-2},k)
\sum_{k_{r-1}+j=k}\binom{k}{j}z_{r-1}^{k_{r-1}}z_r^{j}
\right\} \\
&=\left\{
\sum_{i\geq0}\frac{z_r^{i}}{i!}F(i)
\right\}\left\{
\sum_{k_1,\dots,k_{r-2},k\geq0}
\frac{z_1^{k_1}\cdots z_{r-2}^{k_{r-2}}(z_{r-1} + z_r)^{k}}{k_1!\cdots k_{r-2}!k!}F(k_1,\dots,k_{r-2},k)
\right\} \\
&=\mathcal F(z_r)\mathcal F(z_1,\dots z_{r-2},z_{r-1} + z_r)
\end{align*}
By the equation \eqref{eqn:r=1 for gen. fun. of F}, we get
$$
\mathcal F(z_1,\dots,z_r)
=f_{\geq0}(z_r)\mathcal F(z_1,\dots z_{r-2},z_{r-1} + z_r).
$$
Hence, by using this equation repeatedly, we obtain the claim.
\end{proof}

By Theorem \ref{thm:explicit formula of F}, we know that the special values $F(k_1,\dots,k_r)$ are independent on the pole parts of $f(z)$.
When $f(z)\in\Q[[z]]$, the following theorem holds.

\begin{thm}\label{thm:equivalent of iterated differential}
Let $F(k_1,\dots,k_r)\in\Q$ and put $f(z):=\sum_{k\geq0}\frac{z^k}{k!}a(k)$ with $a(k)\in\Q$.
Then the following two are equivalent:
\begin{enumerate}
\item For $r\geq1$, we have
$$
\sum_{k_1,\dots,k_r\geq0}
\frac{z_1^{k_1}\cdots z_r^{k_r}}{k_1!\cdots k_r!}F(k_1,\dots,k_r)
=f(z_r)f(z_{r-1}+z_r)\cdots f(z_1+\cdots+z_r).
$$
\item For $r\geq1$ and $k_1,\dots,k_r\in\N_0$, we have
$$
F(k_1,\dots,k_r)
=\lim_{z\rightarrow0}
\partial_z^{k_r}\Bigl[ f(z)
\partial_z^{k_{r-1}}\left[ f(z)
\cdots \partial_z^{k_1}[f(z)] \cdots
\right] \Bigr].
$$
\end{enumerate}
\end{thm}
\begin{proof}
At first, we prove (1) from (2).
We consider the map $\phi$ in \eqref{eqn:def of phi by iterated differential}.
Then, by the assumption (2) and the equation \eqref{eqn:recurrence formula of varphi+,- for formal power series}, we have
\begin{align*}
F(k_1,\dots,k_r)
=\lim_{z\rightarrow0}
\phi(d^{k_r}y\cdots d^{k_1}y)
=\lim_{z\rightarrow0}
\phi_+(d^{k_r}y\cdots d^{k_1}y).
\end{align*}
So by Theorem \ref{thm:explicit formula of F}, we obtain (1).

Next, we prove (2) from (1) by induction on $r$.
When $r=1$, it is obvious by the definition of $f(z)$.
Assume that (2) holds for the case $r<r_0(\geq1)$.
When $r=r_0$, by (1), we have
\begin{align*}
&\sum_{k_1,\dots,k_r\geq0}
\frac{z_1^{k_1}\cdots z_r^{k_r}}{k_1!\cdots k_r!}F(k_1,\dots,k_r) \\
&=f(z_r)
\left\{
\sum_{k_1,\dots,k_{r-1}\geq0}
\frac{z_1^{k_1}\cdots z_{r-2}^{k_{r-2}}(z_{r-1} + z_r)^{k_{r-1}}}{k_1!\cdots k_{r-2}!k_{r-1}!}F(k_1,\dots,k_{r-1})
\right\}.
\end{align*}
By comparing coefficients of the term $z_1^{k_1}\cdots z_r^{k_r}$ of both sides, we get
\begin{align*}
F(k_1,\dots,k_r)
&=\sum_{i+j=k_r}\binom{k_r}{i}
F(i)F(k_1,\dots,k_{r-2},k_{r-1}+j).
\intertext{By the induction hypothesis, we have}
&=\sum_{i+j=k_r}\binom{k_r}{i}
\left\{\lim_{z\rightarrow0}
\partial_z^{i}\left[ f(z)\right] \right\}
\left\{\lim_{z\rightarrow0}
\partial_z^{j+k_{r-1}}\left[ f(z)
\cdots \partial_z^{k_1}[f(z)] \cdots
\right]\right\}.
\intertext{By the Leibniz rule of $\partial_z$, we obtain}
&=\lim_{z\rightarrow0}
\partial_z^{k_r}\Bigl[ f(z)
\partial_z^{k_{r-1}}\left[ f(z)
\cdots \partial_z^{k_1}[f(z)] \cdots
\right] \Bigr].
\end{align*}
Hence, we finish the proof.
\end{proof}

We put 
\begin{align}\label{eqn:formal power series as gen. fun. of zetades}
g(z):=\frac{e^z\left\{ (1+z)-e^z \right\}}{(e^z-1)^2} 
\left( =\frac{ze^z}{e^z-1}\cdot \frac{1}{z}\left( \frac{z}{e^z-1} -1 \right) \in\Q[[z]] \right).
\end{align}
We note that this formal series $g(z)$ corresponds with $Z(t_1)|_{t_1=-z}$ in Proposition \ref{prop:explicit formula of gen. fun Z} for $r=1$ in \S\ref{Desingularized MZFs}.
\begin{prop}\label{prop:iterated differential rep. of zetades}
For any $r\geq1$ and $k_1,\dots,k_r\in\N_0$, we have
$$
\zeta^\des_r(-k_1,\dots,-k_r)
=\lim_{z\rightarrow0}
\partial_z^{k_r}\Bigl[ g(z)
\partial_z^{k_{r-1}}\left[ g(z)
\cdots \partial_z^{k_1}[g(z)] \cdots
\right] \Bigr].
$$
\end{prop}
\begin{proof}
By Proposition \ref{prop:explicit formula of gen. fun Z} in \S\ref{Desingularized MZFs} and Theorem \ref{thm:equivalent of iterated differential}, we get the claim.
\end{proof}

\subsection{Definition of $\desLi(\veck)(t)$}
\label{sec:Definition of hypLiveck(t)}
In this subsection, we review the multiple polylogarithms. 
And then, we introduce $\desLi(\veck)(t)$ which is defined by a certain iterated integration. 
At the end of this section, we show the equality for $\desLi(\veck)(t)$ and $\deszeta_r(\veck)$ 
when $(k_1,\dots,k_r)\in\Z^r_{\le0}$ (Proposition \ref{prop:special values of L at non-positive integers}).

The {\it multiple polylogarithms} (MPLs for short) $\Li_{k_1,\dots,k_r}(t)$ with $(k_1,\dots,k_r)\in\Z^r$,
 is the complex analytic function defined by the following series:
\[
  \Li_{k_1,\dots,k_r}(t) := \sum_{0<n_1<\cdots<n_r}
                                              \frac{t^{n_r}}{n_{1}^{k_1}\cdots n_r^{k_r}}
\] 
which converges for $t\in\C$ with $|t|<1$.
The MPL for the case $r=1$, $k=1$ is given by $\Li_1(t)=-\log(1-t)$.
Consult \cite{Zhao2} for many topics related to MPLs.
We will describe some properties of MPLs.
The first to mention is the following {\it iterated integral} expression:
\begin{equation}\label{eqn:iterated integral expression of MPLs}
  \Li_{k_1,\dots,k_r}(t)
   = \int_{0}^t
        \underbrace{ 
         \frac{dt}{t}\circ\dots\circ\frac{dt}{t} 
                           }_{k_r-1}
           \circ\frac{dt}{1-t}
            \circ
             \cdots 
              \circ
               \underbrace{ 
                \frac{dt}{t}\circ\dots\circ\frac{dt}{t} 
                                  }_{k_1-1}
                  \circ\frac{dt}{1-t},
\end{equation}
where $(k_1,\dots,k_r)\in\N^r$.
This yields analytic continuation to a bigger region
\footnote{More precisely, MPLs can be analytically continued to the universal unramified covering $\mathbb{P}^1(\C)\setminus\{0,1,\infty\}$.}.

It is known that $\Li_{-k_1,\dots,-k_r}(t)$ is a rational function of $t$ for $(k_1,\dots,k_r)\in\N_0^r$, 
for instance,
\begin{equation}\label{eqn:PL at non-positive integers}
  \Li_0(t)=\frac{t}{1-t},\qquad
  \Li_{-1}(t)=\frac{t}{(1-t)^2},\qquad
  \Li_{-2}(t)=\frac{t(t+1)}{(1-t)^3}.
\end{equation}

The following differential equation holds for MPLs.
By the definition, one can easily see that 
\begin{equation*}
  \frac{d}{dt}
    \Li_{k_1,\dots,k_r} (t)
  = \begin{cases}
        \frac{1}{t} \Li_{k_1,\dots,k_r}(t) , \quad(k_r\ne 1) \\
        \frac{1}{1-t} \Li_{k_1,\dots,k_{r-1}}(t) , \quad(k_r= 1).
      \end{cases}
\end{equation*}
We note that $\frac{1}{1-t} \Li_{k_1,\dots,k_{r-1}}(t)=\frac{1}{t} \Li_{k_1,\dots,k_{r-1},0}(t)$
for $k_r=1$.
By the above differential equation, 
we see that possible singularities of MPLs for any indices $(k_1,\dots,k_r)\in\Z^r$ are $t=0,\,1$.

Put $\R[[t]]$ to be the algebra of formal power series.
We consider the subalgebra $\mathcal P_{(0,1)}$ of $t\R[[t]]$ defined by
\begin{equation}\label{eqn:def of P(0,1)}
\mathcal P_{(0,1)}:=\{ f(t)\in t\R[[t]]\ |\ \mbox{$f(t)$ converges for $t\in(0,1)$} \}.
\end{equation}
We also set 
$$
\mathcal A:=\mathcal P_{(0,1)}[ \log t ]=\left\{ \sum_{\substack{1\le i \\ 0\le j\le N}}a_{ij}t^i(\log t)^j \,\middle|\, N\in\N_0,\,\forall j\in\N_0,\,\sum_{1\le i}a_{ij}t^i\in\mathcal P_{(0,1)} \right\}.
$$ 

\begin{lem}\label{lem:operators J,D}
For $f \in \mathcal A$, we define
\footnote{These map $J,D$ are also considered in \cite[\S 3.1]{EMS1} as operators on a certain $\C$-algebra of power series.}
$$
J[f]:= \int_0^tf(z)\frac{dz}{z} ,\qquad
D[f]:= t\frac{df}{dt}.
$$
Then $J$ and $D$ form operators on $\mathcal A$, and $J\circ D=D\circ J=\Id$.
\end{lem}
\begin{proof}
It is easy to show $J\circ D=D\circ J=\Id$, so we prove $J[f], D[f] \in \mathcal A$ for any $f \in \mathcal A$.
It is enough to prove $D\left[ g(t)(\log t)^l \right],J\left[ g(t)(\log t)^l \right] \in \mathcal A$ for $g(t)\in \mathcal P_{(0,1)}$ and $l\geq1$.
We have
\begin{align*}
D\left[ g(t)(\log t)^l \right]
=D\left[ g(t) \right](\log t)^l + g(t)D\left[ (\log t)^l \right]
=D\left[ g(t) \right](\log t)^l + g(t)l (\log t)^{l-1}.
\end{align*}
Because we have $D\left[ g(t) \right]\in \mathcal P_{(0,1)}$ by $g(t)\in \mathcal P_{(0,1)}$, we get $D\left[ g(t)(\log t)^l \right]\in\mathcal A$.

On the other hand, by using integration by parts, we have
$$
J\left[ f_1f_2 \right]
=J\left[ f_1 \right]f_2 - J\left[ J[f_1]D[f_2] \right],
$$
for $f_1,f_2\in\mathcal A$.
By using this for $f_1=g(t)$ and $f_2=(\log t)^l$, we calculate
\begin{align*}
J\left[ g(t)(\log t)^l \right]
&=\left[J\left[ g(t) \right](\log t)^l\right]_0^t - J\left[ J[g(t)]D[(\log t)^l] \right] \\
&=J\left[ g(t) \right](\log t)^l - \lim_{t\rightarrow+0}J\left[ g(t) \right](\log t)^l - J\left[ J[g(t)]l (\log t)^{l-1} \right].
\end{align*}
We have $J\left[ g(t) \right]\in \mathcal P_{(0,1)}$ by $g(t)\in \mathcal P_{(0,1)}$ and have $\displaystyle \lim_{t\rightarrow+0}t^{k} (\log t)^l=0$ for $k,l\geq1$, so we get
$$
\lim_{t\rightarrow+0}J\left[ g(t) \right](\log t)^l=0.
$$
Therefore, we get
$$
J\left[ g(t)(\log t)^l \right]=J\left[ g(t) \right](\log t)^l - J\left[ J[g(t)]l (\log t)^{l-1} \right].
$$
Because we have $J\left[ h(t)(\log t)^0 \right]\in\mathcal A$ for $h(t)\in \mathcal P_{(0,1)}$, we inductively obtain $J\left[ g(t)(\log t)^l \right]\in\mathcal A$ for $l\geq1$.
Hence, we finish the proof.
\end{proof}

By the above lemma, we denote $J^{-1}=D$.
By using these operators $J, D$ on $\mathcal A$, we introduce the following elements in $\mathcal A$.
\begin{defn}\label{def:definition of L}
Put $\veck:=(k_1,\dots,k_r)\in\Z^r$ for $r\in\N$.
We define
\footnote{Because $\Li_0(t)$ and $\Li_{-1}(t)$ are in $\mathcal P_{(0,1)}$, 
the element $\desLi(0)(t)$ is in $\mathcal A$.}
$$
\desLi(0)(t):=\Li_0(t)+\log t\cdot\Li_{-1}(t)=\frac{t}{1-t} + \frac{t\log t}{(1-t)^2},
$$
and
\begin{equation*}
\desLi(\veck)(t)
:=J^{k_r}\Bigl[ \desLi(0)(t)
J^{k_{r-1}}\left[ \desLi(0)(t)
\cdots J^{k_1}[\desLi(0)(t)] \cdots
\right] \Bigr].
\end{equation*}
Because we have $\desLi(0)(t)\in\mathcal A$, by Lemma \ref{lem:operators J,D}, 
we get $\desLi(\veck)(t)\in\mathcal A$.
\end{defn}

For $\veck=(k_1,\dots,k_r)\in\Z^r$, we set
\begin{align*}
&\veck':=(k_1,\dots,k_{r-1},k_r-1)\in\Z^r, \\
&\veck^{(j)}:=(k_1,\dots,k_{r-1},k_r-j)\in\Z^r,
\end{align*}
for $j\in\Z$.
Note that $\veck^{(1)}=\veck'$.

\begin{rem}\label{rem:differential equation of L}
By the above definition, we have
\begin{align*}
D[\desLi(\veck)(t)]=\desLi(\veck')(t), \qquad
\desLi(\veck,0)(t)=\desLi(0)(t)\desLi(\veck)(t),
\end{align*}
for $\veck\in\Z^r$.
\end{rem}

\begin{lem}\label{lem:alg. hom. from mathcalH to mathcalA}
We define the $\Q$-linear map\footnote{The Hopf algebra $\mathcal H$ is defined in \eqref{def: Hopf alg. mathcalH}.} $\psi: \mathcal H \rightarrow \mathcal A$ by $\psi(1):=1$ and
$$
\psi(j^{k_1}y \cdots j^{k_r}y):=\desLi(k_r,\dots,k_1)(t),
$$
for $k_1,\dots,k_r\in\Z$.
Then this map $\psi$ forms an algebra homomorphism.\footnote{This lemma is an analogue of \cite[Lemma 3.6]{EMS1} for ordinary MPLs.}
\end{lem}
\begin{proof}
By Definition \ref{def:definition of L}, we have
\begin{equation*}\label{eqn:alg. hom. psi}
\psi(j^{k_1}y \cdots j^{k_r}y)
=J^{k_1}\Bigl[ \desLi(0)(t)
J^{k_{2}}\left[ \desLi(0)(t)
\cdots J^{k_r}[\desLi(0)(t)] \cdots
\right] \Bigr].
\end{equation*}
By the Leibniz rule and the integration by parts and the definition \eqref{eqn:def of shuffle0}, it is clear that the map $\psi$ is an algebra homomorphism.
\end{proof}

The element $\desLi(\veck)(t)\in\mathcal{A}$ defined in Definition \ref{def:definition of L} 
converges to $\deszeta_r(k_1,\dots,k_r)$ when $(k_1,\dots,k_r)\in\Z^r_{\le 0}$ under the limit
$t\rightarrow1-0$.

\begin{prop}\label{prop:special values of L at non-positive integers}
For $r\in\N$ and $\veck:=(k_1,\dots,k_r)\in\Z_{\leq0}^r$, we have
$$
\lim_{t\rightarrow1-0}\desLi(\veck)(t)=\deszeta_r(k_1,\dots,k_r).
$$
\end{prop}
\begin{proof}
Because we have $J^{-1}=D$, we get $J^{-k}=D^k$ for $k\geq0$.
Therefore, for $k_1,\dots,k_r\in\N_0$, we have
\begin{align*}
\desLi(-k_1,\dots,-k_r)(t)
=D^{k_r}\Bigl[ \desLi(0)(t)
D^{k_{r-1}}\left[ \desLi(0)(t)
\cdots D^{k_1}[\desLi(0)(t)] \cdots
\right] \Bigr].
\end{align*}
Consider changing variables $t=e^z$. Then we have
$$
D=t\frac{d}{dt}=e^z\frac{dz}{dt}\cdot\frac{d}{dz}=\frac{d}{dz}=\partial_z,
$$
and, by \eqref{eqn:PL at non-positive integers}, we get
$$
\desLi(0)(t)
=\frac{t}{1-t} + \log t \frac{t}{(1-t)^2}
=\frac{e^z}{1-e^z} + \frac{ze^z}{(1-e^z)^2}
=g(z).
$$
Here, $g(z)$ is given in \eqref{eqn:formal power series as gen. fun. of zetades}.
So we get
$$
\desLi(-k_1,\dots,-k_r)(t)
=\partial_z^{k_r}\Bigl[ g(z)
\partial_z^{k_{r-1}}\left[ g(z)
\cdots \partial_z^{k_1}[g(z)] \cdots
\right] \Bigr].
$$
When $t\rightarrow 1$, we have $z\rightarrow0$, so by Proposition \ref{prop:iterated differential rep. of zetades}, we obtain
$$
\lim_{t\rightarrow1}\desLi(-k_1,\dots,-k_r)(t)=\zeta^\des_r(-k_1,\dots,-k_r).
$$
Hence, we finish the proof.
\end{proof}

In Theorem \ref{thm:holomorphic of desLi} in \S\ref{sec: Main results}, 
we will generalize this proposition.

\section{Linear combinations of MPLs}
\label{sec: Linear combinations of MPLs}
In this section, we consider certain finite linear combinations $Z(\veck;t)$ of MPLs to express $\desLi(\veck)(t)$ as a finite linear combination of MPLs (Definition \ref{def:infinite sum M of Mq}). 
In \S\ref{subsec: Definition of M(veck) and its properties}, we introduce the function $\deszeta_r(s_1,\dots,s_r)(t)$ and explain its properties.
By using these, we show some properties of $Z(\veck;t)$.
In \S\ref{subsec: On the equality of Z(veck;t) and desLi(veck)(t)}, 
we show that $Z(\veck;t)$ and $\desLi(\veck)(t)$ are equal (Theorem \ref{thm:M(veck)=hypLi}).

\subsection{Definition of $Z(\veck)$ and its properties}
\label{subsec: Definition of M(veck) and its properties}
In this subsection, we consider a certain finite linear combination of MPLs.
To this end, we first consider certain multiple zeta functions 
and their properties which will be employed in later sections.

\begin{defn}[cf. {\cite{EMT}}]
  For $0< t < 1$ and any $\vecs=(s_1,\dots,,s_r)\in\C^r$,
  we set
  \[
    \zeta_r^{\shuffle}(\vecs;t) = \zeta_r^{\shuffle}(s_1,\dots,s_r;t)
     := \sum_{0<n_1<\cdots<n_r}
           \frac{ t^{n_r} }
                  { n_1^{s_1}\cdots n_r^{s_r} }.
  \]
\end{defn}

By \cite[Corollary 2.2]{EMT}, 
we see that $\zeta_r^{\shuffle}(\vecs;t)$  is entire on the whole space $\C^r$.
We also denote $\zeta_r^{\shuffle}(\vecs;t)$ as $\zeta_r^{\shuffle}((s_j);t)$.

\begin{rem}\label{rem:shufflezeta=Li at integer points}
By definition, we have
\begin{equation*}\label{eqn:shufflezeta=Li at integer points}
 \zeta_r^{\shuffle}(k_1,\dots,k_r;t)=\Li_{k_1,\dots,k_r}(t),
\end{equation*}
for $k_1,\dots,k_r\in \Z$. It is clear that for $(s_1,\dots,s_r)\in\mathcal{D}$,
\begin{equation*}
  \lim_{t\rightarrow1-0}\zeta^{\shuffle}_r(s_1,\dots,s_r;t) = \zeta_{r}(s_1,\dots,s_r).
\end{equation*}
However, we note that $\lim_{t\rightarrow1-0}\zeta^{\shuffle}_r(s_1,\dots,s_r;t)$ diverges for
$(s_1,\dots,s_r)\in\C^r\setminus\mathcal{D}$.
\end{rem}

We consider the following generating function. 
For $c\in\R$ and $0<t<1$, we put
\begin{equation*}
 \begin{split}
  \tilde{\mathfrak{H}}_r(t_1,\dots,t_r;c;t)
   &:= \prod_{j=1}^{r}
            \left(
               \frac{1}{\frac{1}{t}\exp(\sum_{k=j}^{r}t_k)-1}
                - \frac{c}{\frac{1}{t}\exp(c\sum_{k=j}^{r}t_k)-1}
            \right).
 \end{split}
\end{equation*}

\begin{defn}[cf. {\cite[Definition 3.1]{FKMT1}}]\label{def:definition of deszeta}
  For $0<t<1$, $s_1,\dots,s_r\in\C\setminus\Z$, 
  we define
  \begin{equation}\label{eqn:definition of sup mzf}
      \deszeta_r(s_1,\dots,s_r)(t) 
       := \lim_{ \substack{c \rightarrow 1 \\ c\ne 1} }
               \frac{1}{(1-c)^r}
                C(\vecs)
                  \int_{\mcalc_t^r}
                   \tilde{\mathfrak{H}}_r(t_1,\dots,t_r;c;t)
                    \prod_{k=1}^{r}t_k^{s_{k}-1}dt_k,
  \end{equation}
  where 
  $\mcalc_t$ is a certain Hankel's contour, that is, the path consisting of the positive axis (top side), 
  a circle around the origin of radius $\epsilon_t$ (with $0<\epsilon_t<|\log t|$), 
  and the positive real axis (bottom side). 
  The symbol $C(\vecs)$ is defined in \eqref{eqn: coefficient of contour integration}.
\end{defn}
\begin{rem}
We note that the limit $\lim_{t\rightarrow 1-0}\deszeta_r(s_1,\dots,s_r)(t)$ does not always exist.
For instance, consider the case $r=1$:
we have 
\begin{equation*}
  \deszeta_1(s)(t) 
   = (1-s)\zeta^{\shuffle}(s;t). 
\end{equation*}
Thus, $\lim_{t\rightarrow 1-0}\deszeta_r(s)(t)$ diverges for $s\in\C\setminus\mathcal{D}$.
\end{rem}

The following two propositions are analogs of 
Proposition \ref{prop:deszeta can be written as sum of mzfs} and 
\ref{prop:recurrence relation of deszeta}, and can be proved exactly in the same way.

\begin{prop}\label{prop:deszeta can be written as sum of shuffle-mzfs}
  For $s_1,\dots,s_r\in\C$, and $0<t< 1$, we have the following equality 
  between meromorphic functions of the complex variables $(s_1,\dots,s_r)$;
  \[
    \deszeta_r(s_1,\dots,s_r)(t) 
     = \sum_{ \substack{\vecl=(l_j)\in\N_0^r \\ \vecm(m_j)\in\Z^r \\ |\vecm|=0} }
          \almr 
           \left( \prod_{j=1}^{r} (s_j)_{l_j}\right)
            \zeta^\shuffle_r(s_1+m_1,\dots,s_r+m_r;t),
  \]
  where the coefficient $\almr$ is defined in \eqref{eqn:the seq of coeff of deszeta}, and the symbol $(s)_k$ is the Pochhammer symbol (see Proposition \ref{prop:deszeta can be written as sum of mzfs} for the definition).
\end{prop}
\begin{prop}\label{prop:recurrence relation of t-deszeta}
  For $s_1,\dots,s_{r-1}\in\C$, $k\in\N_0$ and $0<t<1$, we have
  \begin{equation*}
    \deszeta_r(s_1,\dots,s_{r-1},-k)(t) 
      =  \sum_{i=0}^k
             \binom{k}{i}
               \deszeta_{r-1}(s_1,\dots,s_{r-1}-k+i)(t)
                 \deszeta_1(-i)(t).
  \end{equation*}
\end{prop}

\begin{defn}\label{def:definition of Mq}
We define\footnote{The algebra $\mathcal P_{(0,1)}$ is defined in \eqref{eqn:def of P(0,1)}.} the family $\{Z_q(\veck)\}_{q\geq0,r\geq1,\veck\in\Z^r}$ in $\mathcal P_{(0,1)}$ by, for $r\geq1$, $\veck\in\Z^r$,
$$
Z_0(\veck)
:=Z_0(\veck;t)
:=\deszeta_r(\veck)(t),
$$
and for $q\geq1$, $r\geq1$ and $\veck\in\Z^r$,
\begin{equation}\label{eqn:definition of Mq}
Z_q(\veck)
:=Z_q(\veck;t)
:=\sum_{i+j=q}(-1)^j\binom{q}{i}D^i\left[ Z_0(\veck^{(j)}) \right].
\end{equation}
\end{defn}
By Remark \ref{rem:shufflezeta=Li at integer points} and Proposition \ref{prop:deszeta can be written as sum of shuffle-mzfs}, we have
\begin{align}\label{eqn:definition of M0}
Z_0(\veck)
:=Z_0(\veck;t)
&:= \sum_{ \substack{\vecl=(l_j)\in\N_0^r \\ \vecm(m_j)\in\Z^r \\ |\vecm|=0} }
         \almr 
          \left( \prod_{j=1}^{r} (k_j)_{l_j}\right)
           \Li_{\veck+\vecm}(t),
\end{align}
so the above definition is well-defined as elements in $\mathcal P_{(0,1)}$.
In Appendix \ref{sec:Explicit formula of Mq(veck)}, we give explicit formulae of $Z_q(\veck)$ in terms of $\Li_{\veck+\vecm}(t)$ for $\vecm\in\Z^r$.

The following recurrence formula holds for the above element $Z_q(\veck)$.
\begin{prop}\label{prop:recurrence relation of Mq}
For $q\geq1$, $r\geq1$ and $\veck\in\Z^r$, we have
$$
Z_q(\veck)
=D\left[ Z_{q-1}(\veck) \right] - Z_{q-1}(\veck').
$$
\end{prop}
\begin{proof}
By using the following Lemma \ref{lem:recurrence formula of the sum of binomial coeff.} for $f(i,j)=(-1)^jD^i\left[ Z_0(\veck^{(j)}) \right]$, we have
\begin{align*}
&Z_q(\veck) \\ 
&=\sum_{i+j=q-1}(-1)^j\binom{q-1}{i}D^{i+1}\left[ Z_0(\veck^{(j)}) \right]
- \sum_{i+j=q-1}(-1)^j\binom{q-1}{i}D^i\left[ Z_0(\veck^{(j+1)}) \right] \\
&=D\left[\sum_{i+j=q-1}(-1)^j\binom{q-1}{i}D^{i}\left[ Z_0(\veck^{(j)}) \right] \right] 
- \sum_{i+j=q-1}(-1)^j\binom{q-1}{i}D^i\left[ Z_0\left( (\veck')^{(j)} \right) \right] \\
&=D\left[ Z_{q-1}(\veck) \right] - Z_{q-1}(\veck').
\end{align*}
Hence, we obtain the claim.
\end{proof}

\begin{lem}\label{lem:recurrence formula of the sum of binomial coeff.} 
Let $f$ be a map on $\N_0^2$.
For $q\geq1$, we have
\begin{align*}
\sum_{i+j=q}\binom{q}{i}f(i,j)
=\sum_{i+j=q-1}\binom{q-1}{i}f(i+1,j) + \sum_{i+j=q-1}\binom{q-1}{i}f(i,j+1).
\end{align*}
\end{lem}
\begin{proof}
Because we have the recurrence relation of binomial coefficients
$$
\binom{q}{i}=\binom{q-1}{i-1} + \binom{q-1}{i}
$$
for $1\leq i\leq q-1$ and we have
$$
\binom{q}{0}=\binom{q}{q}=1,
$$
so we obtain the claim.
\end{proof}

For our simplicity, we sometimes denote
\begin{equation}\label{eqn:definition of M0 by p(veck)}
Z_0(\veck)
=\sum_{\vecm\in\Z^r}p_\vecm(\veck)\Li_{\veck+\vecm}(t).
\end{equation}
Here, $p_\vecm(\veck)$ means $p_\vecm(x_1,\dots,x_r)|_{x_i=k_i}$ where a family $\{p_\vecm(x_1,\dots,x_r)\}_{\vecm\in\Z^r}$ in $\Q[x_1,\dots,x_r]$ with $p_\vecm(x_1,\dots,x_r)=0$ except for a finite number of $\vecm\in\Z^r$.
We denote $\deg_{x_r}(p_\vecm(x_1,\dots,x_r))$ by the degree of the polynomial $p_\vecm(x_1,\dots,x_r)$ in $x_r$, and for $r\geq1$, we put
$$
d_r:=\max\{ \deg_{x_r}(p_\vecm(x_1,\dots,x_r))\ |\ \vecm\in\Z^r \}.
$$
\begin{prop}\label{prop:Mq(k)=0}
For $q\geq d_r+1$ and for $\veck\in\Z^r$, we have
$$
Z_q(\veck)=0.
$$
\end{prop}
\begin{proof}
If we have $Z_q(\veck)=0$ for $q=d_r+1$, we inductively get $Z_q(\veck)=0$ for $q>d_r+1$ by Proposition \ref{prop:recurrence relation of Mq}.
So it is sufficient to prove $Z_q(\veck)=0$ for $q=d_r+1$.
By the expression \eqref{eqn:definition of M0 by p(veck)} and the definition \eqref{eqn:definition of Mq}, we have the following representation:
$$
Z_q(\veck)
=\sum_{\vecm\in\Z^r}c_q(\veck;\vecm)\Li_{\veck^{(q)}+\vecm}(t),
$$
where the symbol $c_q(\veck;\vecm)$ means $c_q(x_1,\dots x_r;\vecm)|_{x_i=k_i}$ with
$$
c_q(x_1,\dots x_r;\vecm):=\sum_{i+j=q}(-1)^j\binom{q}{i}p_\vecm(x_1,\dots,x_{r-1},x_r-j) \ \left(\in \Q[x_1,\dots,x_r]\right).
$$
We note that
\begin{equation}\label{eqn:max degree of c0}
\max\{ \deg_{x_r}(c_0(x_1,\dots x_r;\vecm))\ |\ \vecm\in\Z^r \}=d_r.
\end{equation}
By Proposition \ref{prop:recurrence relation of Mq}, we calculate
\begin{align*}
c_q(x_1,\dots x_r;\vecm)|_{x_i=k_i}
&=c_q(\veck;\vecm)
=c_{q-1}(\veck;\vecm) - c_{q-1}(\veck';\vecm) \\
&=\left\{ c_{q-1}(x_1,\dots x_r;\vecm) - c_{q-1}(x_1,\dots x_{r-1},x_r-1;\vecm) \right\}|_{x_i=k_i}
\end{align*}
for $q\geq1$ and for any $k_1,\dots,k_r\in\Z$.
Therefore, we have
\begin{align}\label{eqn:deg of cq=deg of c(q-1) - 1}
&\deg_{x_r}(c_q(x_1,\dots x_r;\vecm)) \\
&=\left\{\begin{array}{ll}
\deg_{x_r}(c_{q-1}(x_1,\dots x_r;\vecm)) - 1 & (\deg_{x_r}(c_{q-1}(x_1,\dots x_r;\vecm))\geq1) ,\\
0 & (\deg_{x_r}(c_{q-1}(x_1,\dots x_r;\vecm))=0).
\end{array}\right. \nonumber
\end{align}
By using \eqref{eqn:max degree of c0} and \eqref{eqn:deg of cq=deg of c(q-1) - 1}, we get
$$
\deg_{x_r}(c_{d_r}(x_1,\dots x_r;\vecm)) = 0
$$
for any $\vecm\in\Z^r$, that is, we have $c_{d_r}(x_1,\dots x_r;\vecm)\in \Q[x_1,\dots,x_{r-1}]$.
Hence, we obtain
$$
c_{d_r+1}(x_1,\dots x_r;\vecm)=0
$$
for any $\vecm\in\Z^r$, so we have $Z_{d_r+1}(\veck)=0$.
\end{proof}

\begin{defn}\label{def:infinite sum M of Mq}
For $\veck\in\Z^r$, we define the element $M(\veck)\in \mathcal A(=\mathcal P_{(0,1)}[ \log t ])$ by
\begin{equation}\label{eqn:def of M(veck)}
Z(\veck)
:=Z(\veck;t)
:=\sum_{q\geq0}\frac{(-\log t)^q}{q!}Z_q(\veck).
\end{equation}
\end{defn}
\begin{rem}
By Proposition \ref{prop:Mq(k)=0}, the right-hand side of the equation \eqref{eqn:def of M(veck)} is a finite sum, so the above definition is well-defined as the element in $\mathcal A$.
\end{rem}

\begin{exa}\label{ex: example of M in one variable}
  We consider the case $r=1$.
  By the definition \eqref{eqn:definition of M0}, 
  $Z_0(k)$ is presented as follows:
  \[
      Z_0(k) = (1-k)\Li_k(t).
  \]
  $Z_1(k)$ is calculated to be
  \[
      Z_1(k) = -(\log t)\, \Li_{k-1}(t).
  \]
  By Proposition \ref{prop:Mq(k)=0}, we know $Z_q(k)=0$ for $q\ge2$.
  Thus, we obtain 
  \begin{equation}\label{eqn: explicit formula of M(k;t)}
      Z(k) = Z_0(k) - Z_1(k) = (1-k)\Li_k(t) + (\log t)\Li_{k-1}(t)
  \end{equation}
  for $k\in\Z$. 
  \end{exa}

We next show that $Z(\veck)$ satisfies the differential equation which holds for MPLs and 
$\desLi(\veck)(t)$ in Definition \ref{def:definition of L}.
\begin{thm}\label{thm:differential equation of M}
For $\veck\in\Z^r$, we have
$$
D\left[ Z(\veck) \right]=Z(\veck').
$$
\end{thm}
\begin{proof}
We have
\begin{align*}
&D\left[ Z(\veck) \right]-Z(\veck') 
=D\left[ \sum_{q\geq0}\frac{(-\log t)^q}{q!}Z_q(\veck) \right]
- \sum_{q\geq0}\frac{(-\log t)^q}{q!}Z_q(\veck') \\
&=D\left[Z_0(\veck) \right] - Z_0(\veck')
+D\left[ \sum_{q\geq1}\frac{(-\log t)^q}{q!}Z_q(\veck) \right]
- \sum_{q\geq1}\frac{(-\log t)^q}{q!}Z_q(\veck').
\intertext{By Definition \ref{def:definition of Mq}, we get}
&=D\left[Z_0(\veck) \right] - Z_0(\veck')
+D\left[ \sum_{q\geq1}\frac{(-\log t)^q}{q!}\sum_{i+j=q}(-1)^j\binom{q}{i}D^i\left[ Z_0(\veck^{(j)}) \right] \right] \\
&\quad - \sum_{q\geq1}\frac{(-\log t)^q}{q!}\sum_{i+j=q}(-1)^j\binom{q}{i}D^i\left[ Z_0(\veck^{(j+1)}) \right]. \\
\intertext{By applying the Leibniz rule to the third term, we calculate}
&=D\left[Z_0(\veck) \right] - Z_0(\veck')
- \sum_{q\geq1}\frac{(-\log t)^{q-1}}{(q-1)!}\sum_{i+j=q}(-1)^j\binom{q}{i}D^i\left[ Z_0(\veck^{(j)}) \right]  \\
&\quad + \sum_{q\geq1}\frac{(-\log t)^q}{q!}\sum_{i+j=q}(-1)^j\binom{q}{i}D^{i+1}\left[ Z_0(\veck^{(j)}) \right] \\
&\qquad + \sum_{q\geq1}\frac{(-\log t)^q}{q!}\sum_{i+j=q}(-1)^{j+1}\binom{q}{i}D^i\left[ Z_0(\veck^{(j+1)}) \right].
\intertext{By applying Lemma \ref{lem:recurrence formula of the sum of binomial coeff.} for $f(i,j)=(-1)^jD^i\left[ Z_0(\veck^{(j)}) \right]$ to the fourth and the fifth terms, we get}
&=D\left[Z_0(\veck) \right] - Z_0(\veck')
- \sum_{q\geq1}\frac{(-1)^{q-1}}{(q-1)!}(\log t)^{q-1}\sum_{i+j=q}(-1)^j\binom{q}{i}D^i\left[ Z_0(\veck^{(j)}) \right]  \\
&\quad + \sum_{q\geq1}\frac{(-1)^q}{q!}(\log t)^q\sum_{i+j=q+1}(-1)^j\binom{q+1}{i}D^{i}\left[ Z_0(\veck^{(j)}) \right].
\intertext{By rearranging the first, second, and fourth terms, and by changing variables of summation of the third term from $q\geq1$ to $q\geq0$, we obtain}
&=- \sum_{q\geq0}\frac{(-1)^{q}}{q!}(\log t)^{q}\sum_{i+j=q+1}(-1)^j\binom{q+1}{i}D^i\left[ Z_0(\veck^{(j)}) \right]  \\
&\quad + \sum_{q\geq0}\frac{(-1)^q}{q!}(\log t)^q\sum_{i+j=q+1}(-1)^j\binom{q+1}{i}D^{i}\left[ Z_0(\veck^{(j)}) \right]
=0.
\end{align*}
Hence, we finish the proof.
\end{proof}

\subsection{On the equality of $Z(\veck;t)$ and $\desLi(\veck)(t)$}
\label{subsec: On the equality of Z(veck;t) and desLi(veck)(t)}
In this subsection, we show that $\desLi(\veck)(t)$ can be expressed as a certain `` linear'' combinations of MPLs (Theorem \ref{thm:M(veck)=hypLi}).

By the definition of $Z_0(\veck)$ and Proposition \ref{prop:recurrence relation of t-deszeta}, we have the following equation:
\begin{equation}\label{eqn:recurrence formula of M}
Z_0(\veck,-k_{r+1})
=\sum_{i+j=k_{r+1}}\binom{k_{r+1}}{i} Z_0(-i)Z_0(\veck^{(j)}).
\end{equation}
Here, $r\in\N$, $\veck\in\Z^r$ and $k_{r+1}\in\N_0$.

\begin{prop}\label{prop:M(k,0)=M(0)M(k)}
For $r\in\N$ and $\veck\in\Z^r$, we have
$$
Z(\veck,0) = Z(0)Z(\veck).
$$
\end{prop}
\begin{proof}
Let $q\geq0$. By definition \eqref{eqn:definition of Mq} of $Z_q(\veck)$ and by the condition \eqref{eqn:recurrence formula of M}, we have
\begin{align*}
&Z_q(\veck,0) 
=\sum_{i+j=q}(-1)^j\binom{q}{i}D^i\left[ Z_0(\veck,-j) \right] \\
&=\sum_{i+j=q}(-1)^j\binom{q}{i}D^i\left[ \sum_{b+d=j}\binom{j}{b} Z_0(-b)Z_0(\veck^{(d)}) \right].
\intertext{By the Leibniz rule, we calculate}
&=\sum_{i+j=q}(-1)^j\binom{q}{i} \sum_{b+d=j}\binom{j}{b}
\left\{
\sum_{a+c=i}\binom{i}{a}
D^a\left[ Z_0(-b) \right]
D^c\left[ Z_0(\veck^{(d)}) \right]
\right\} \\
&=\sum_{a+c+b+d=q}(-1)^{b+d}
\frac{q!}{b!d!a!c!}
D^a\left[ Z_0(-b) \right]
D^c\left[ Z_0(\veck^{(d)}) \right] \\
&=\sum_{i+j=q}\binom{q}{i}
\left\{
\sum_{a+b=i}(-1)^{b}\binom{i}{a}
D^a\left[ Z_0(-b) \right]
\right\}\left\{
\sum_{c+d=j}(-1)^{d}\binom{j}{c}
D^c\left[ Z_0(\veck^{(d)}) \right]
\right\}.
\intertext{By using definition \eqref{eqn:definition of Mq} of $Z_q(\veck)$ again, we get}
&=\sum_{i+j=q}\binom{q}{i} Z_i(0)Z_j(\veck).
\end{align*}
Therefore, we obtain
\begin{align*}
Z(\veck,0)
&=\sum_{q\geq0}\frac{(-\log t)^q}{q!}Z_q(\veck,0)
=\sum_{q\geq0}\frac{(-\log t)^q}{q!} \sum_{i+j=q}\binom{q}{i} Z_i(0)Z_j(\veck) \\
&=\left\{
\sum_{i\geq0}\frac{(-\log t)^i}{i!}Z_i(0)
\right\}\left\{
\sum_{j\geq0}\frac{(-\log t)^j}{j!}Z_j(\veck)
\right\}
=Z(0)Z(\veck).
\end{align*}
Hence, we finish the proof.
\end{proof}

By using the above proposition, we obtain the following.
\begin{thm}\label{thm:M(veck)=hypLi}
For $r\in\N$ and $\veck\in\Z^r$, we have
$$
\desLi(\veck)(t) = Z(\veck).
$$
\end{thm}
\begin{proof}
We prove this claim by induction on $r$.
When $r=1$, by Definition \ref{def:definition of L} and \eqref{eqn: explicit formula of M(k;t)}, we have
\begin{equation}\label{eqn:M(0)=Ldes(0)(t)}
\desLi(0)(t)
=\Li_0(t)+\log t\cdot\Li_{-1}(t)
=Z(0).
\end{equation}
So by Remark \ref{rem:differential equation of L} and Theorem \ref{thm:differential equation of M}, we have
$$
\desLi(k_1)(t)
=J^{k_1}\left[ \desLi(0)(t) \right]
=J^{k_1}\left[ Z(0) \right]
=Z(k_1).
$$
Hence the claim holds for $r=1$.
Assume that the claim holds for $r_0\in\N$.
Put $\veck=(k_1,\dots,k_{r_0+1})$ with $k_1,\dots,k_{r_0+1}\in\Z$.
Similarly to the case of $r=1$, we have
$$
\desLi(\veck)(t)
=J^{k_{r_0+1}}\left[ \desLi(k_1,\dots,k_{r_0},0)(t) \right]
=J^{k_{r_0+1}}\left[ \desLi(0)(t)\desLi(k_1,\dots,k_{r_0})(t) \right].
$$
On the other hand, by using Theorem \ref{thm:differential equation of M} and Proposition \ref{prop:M(k,0)=M(0)M(k)}, we get$$
Z(\veck)
=J^{k_{r_0+1}}\left[ Z(k_1,\dots,k_{r_0},0) \right]
=J^{k_{r_0+1}}\left[ Z(0)Z(k_1,\dots,k_{r_0}) \right].
$$
Hence, by \eqref{eqn:M(0)=Ldes(0)(t)} and the induction hypothesis, we obtain the claim.
\end{proof}

\section{Main results}
\label{sec: Main results}

In this section, we introduce one-parameterized desingularized MZFs 
(Definition \ref{defn: the definition of hatdeszeta}).
After that, we prove that the shuffle type formula holds for special values of desingularized MZFs 
at integer points (Theorem \ref{cor:shuffle product of deszeta}).
We also give some examples.

\subsection{Definition of $\hatdeszeta_r(\vecs;t)$}
\label{sec: Definition of hypzeta(vecs;t)}

We introduce a new function $\hatdeszeta_r(\vecs)(t)$, and we investigate several properties of this function. More precisely, we show that $\hatdeszeta_r(\vecs)(t)$ converges to $\deszeta_r(\vecs)$
under the limit $t\rightarrow1-0$.
\begin{defn}\label{defn: the definition of hatdeszeta}
For $s_1,\dots,s_r\in\C\setminus\Z$, and $0<t<1$, we define
\begin{align}\label{eqn: the definition of hyperposed zeta function}
&\hatdeszeta_r(s_1,\dots,s_r)(t)
:= \prod_{k=1}^{r}
	\frac{1}{(e^{2\pi i s_k}-1)\Gamma(s_k)} \\ 
&\hspace{5mm}\cdot
\int_{\mcalc^r}
\prod_{j=1}^{r}
\left(\frac{1}{\frac{1}{t}\exp(\sum_{k=j}^{r}x_k)-1}
+ \frac{(\log t-\sum_{k=j}^{r}x_k)\frac{1}{t}\exp(\sum_{k=j}^{r}x_k)}{\{\frac{1}{t}\exp(\sum_{k=j}^{r}x_k)-1\}^2}
\right) 
\prod_{k=1}^{r}x_k^{s_{k}-1}dx_k. \nonumber
\end{align}
\end{defn}

We note that the convergence of the right-hand side of \eqref{eqn: the definition of hyperposed zeta function} 
can be justified by the following proposition.

\begin{prop}\label{prop:holomorphic of t-deszeta}
For any $t\in(0,1)$, the function $\hatdeszeta_r(s_1,\dots,s_r)(t)$ can be analytically continued to $\C^{r}$ as an entire function in $(s_1,\dots,s_r)\in\C^{r}$.
Moreover, for any $(s_1,\dots,s_r)\in\C^r$, we have
$$
\lim_{t\rightarrow1-0}\hatdeszeta_r(s_1,\dots,s_r)(t)=\deszeta_r(s_1,\dots,s_r).
$$
\end{prop}

For the proof of Proposition 4.2, 
we prepare two lemmas.
Let $\mcaln(\epsilon):=\{ z\in\C \,|\, |z|<\epsilon \}$ and 
$\mcals(\theta):=\{ z\in\C \,|\, |\arg z|\le \theta \}$.
Then one can easily obtain:
\begin{lem}[{\cite[Lemma 3.5]{FKMT1}}] 
\label{lem: preparation1 for the proof of holomorphy of hypzeta}
  There exist $\epsilon>0$ and $0<\theta<\pi/2$ such that 
  \[
     \sum_{k=j}^{r} x_k \in \mcaln(1)\cup\mcals(\theta) 
  \]
  for any $x_j,\dots,x_r\in\mcalc$ ($1\le j \le r$), 
  where $\mcalc$ is the Hankel contour involving a circle around the origin of radius $\epsilon$.
\end{lem}


\begin{lem}[{cf. \cite[Lemma 3.6]{FKMT1}}] 
\label{lem: preparation2 for the proof of holomorphy of hypzeta}
  For any $t\in(0,1)$ and $y\in\mcaln(1)\cup\mcals(\theta)$, 
  there exists a constant $A>0$ independent of $t$ such that 
  \begin{equation*}
    \left| \frac{1}{\frac{1}{t}e^y-1} + \frac{(\log t - y)\frac{1}{t}e^y}{(\frac{1}{t}e^y-1)^2} \right|
    < Ae^{-\Re y/2}.
  \end{equation*}
\end{lem}
\begin{proof}
  We set 
  \[
     F(t;y):=\frac{1}{\frac{1}{t}e^y-1} + \frac{(\log t - y)\frac{1}{t}e^y}{(\frac{1}{t}e^y-1)^2}.
  \]
  We first note that 
  \begin{equation*}\label{eqn: the expansion of the integrand of hypzeta}
    \frac{1}{t}e^y-1 = y-\log t + \frac{1}{2}(y-\log t)^2 + O\big((y-\log t)^2\big), \quad
    (y\rightarrow \log t).
  \end{equation*}
  Thus, we get
  \begin{align*}
    F(t;y) &= \frac{ \frac{1}{t}e^y-1 + (\log t - y)\frac{1}{t}e^y }
                              { (\frac{1}{t}e^y-1)^2 } 
               = \frac{ -\frac{1}{2}(y-\log t)^2 + O\big((y-\log t)^2\big) }
                              { (y-\log t)^2 + O\big((y-\log t)^2\big) }.
  \end{align*}
  If $\log t<-1$, that is, $\log t\not\in \mcaln(1)\cup\mcals(\theta)$, then $F(t;y)$ is holomorphic for all $y\in\mcaln(1)\cup\mcals(\theta)$.
  One can also see that $F(t;y)$ has the limit value when $y\rightarrow \log t$ for $-1\le \log t <0$.
  Hence, there exists $C>0$ such that for any $y\in\mcaln(1)$,  
  \[
       |F(t;y)| < C.
  \]  
  We next consider the case $y\in\mcals(\theta)\setminus\mcaln(1)$. 
  Note that there does not exist $t\in(0,1)$ such that $\log t = y$ in this case.
  Thus, there exists $A_1$, $A_2>0$ such that 
  \begin{align*}
    \left| \frac{1}{ \frac{1}{t}e^y-1 } \right| < A_1e^{-\Re y/2}, \quad
    \left| \frac{(\log t - y)\frac{1}{t}e^y}{(\frac{1}{t}e^y-1)^2} \right| < A_2e^{-\Re y/2}.
  \end{align*}
  Therefore, we obtain the claim by putting $A:=A_1+A_2$.
\end{proof}

\noindent{\it Proof of Proposition \ref{prop:holomorphic of t-deszeta}}.
 We use the notation used in Lemma \ref{lem: preparation1 for the proof of holomorphy of hypzeta} and \ref{lem: preparation2 for the proof of holomorphy of hypzeta}.
  We put 
  \begin{align*}
    G(x_1,\dots,x_r) 
    := A^r \prod_{j=1}^r
                    \exp\left(
                      -\Re\left(
                          \sum_{k=j}^r  t_k/2
                              \right)
                            \right) 
    = A^r \prod_{j=1}^r
                   \exp\left(
                     -\Re\left(t_k \frac{k(k+1)}{4} \right)
                           \right).                   
  \end{align*}
  Then, one can easily show that 
  \begin{align}
    \label{eqn: dominated function of F}
    &\left| \prod_{j=1}^r F\left(t;\sum_{k=j}^r x_k\right) \right| < G(x_1,\dots,x_r), 
                                                                                                               \quad(x_1,\dots,x_r\in\mcalc), \\
    \label{eqn: convergence of dominated function}
    &\int_{\mcalc^r} G(x_1,\dots,x_r) \prod_{k=1}^r|t_k^{s_k-1}dt_k| < \infty.
  \end{align}
  Since the integral on the right-hand side of \eqref{eqn: the definition of hyperposed zeta function}
  is holomorphic for all $s_1,\dots,s_r\in\C$, 
  we see that $\hatdeszeta_r(s_1,\dots,s_r)(t)$ can be meromorphically continued to $\C^r$.
  We also see that its possible singularities are located on hyperplanes $s_k=l_k\in\N$.
  For $s_k=l_k\in\N$, one can show that 
  the integration of \eqref{eqn: the definition of hyperposed zeta function} is zero 
  by using the residue theorem.
  Thus, $\hatdeszeta_r(s_1,\dots,s_r)(t)$ has no singularity on $s_k=l_k$.
  In other words, $\hatdeszeta_r(s_1,\dots,s_r)(t)$ is entire on $\C^r$.
  The latter part of the claim is proved by Lebesgue's convergence theorem, 
  \eqref{eqn: dominated function of F}, and \eqref{eqn: convergence of dominated function}. 
  \hfill$\Box$

\subsection{Shuffle product of desingularized MZFs at integer points}
In this subsection, we show that the shuffle product formula holds for products of special values 
at integer points of desingularized MZFs.
At the end of this section, explicit formulas for the product of depth 1 and depth 1, 
and that of depth 1 and depth 2, are given.

\begin{thm}\label{thm:hypzeta=M(veck)}
For any $r\geq1$, $\veck=(k_1,\dots,k_r)\in\Z^r$ and for $t\in(0,1)$, we have
$$
\hatdeszeta_r(k_1,\dots,k_r)(t)=Z(\veck).
$$
\end{thm}
\begin{proof}
We set $[r]:= \{ 1,2,\dots,r \}$ for $r\in\N$.
We first calculate the integrand of the right-hand side in the equation \eqref{eqn: the definition of hyperposed zeta function}:
\begin{align*}
&\prod_{j=1}^r\left( 
\frac{1}{\frac{1}{t}\exp(\sum_{k=j}^rx_k)-1} 
+ \frac{(\log t-\sum_{k=j}^rx_k)\frac{1}{t}\exp(\sum_{k=j}^rx_k)}{\{ \frac{1}{t}\exp(\sum_{k=j}^rx_k)-1 \}^2} 
\right) \\
&=\prod_{j=1}^r\left( 
\frac{1}{\frac{1}{t}\exp(\sum_{k=j}^rx_k)-1} 
- \frac{\frac{1}{t}\sum_{k=j}^rx_k\exp(\sum_{k=j}^rx_k)}{\{ \frac{1}{t}\exp(\sum_{k=j}^rx_k)-1 \}^2} 
+\frac{\frac{1}{t}\exp(\sum_{k=j}^rx_k)}{\{ \frac{1}{t}\exp(\sum_{k=j}^rx_k)-1 \}^2}\log t 
\right) \\
&=\sum_{J\subset[r]}
\prod_{j\in J\setminus[r]}\left( 
\frac{1}{\frac{1}{t}\exp(\sum_{k=j}^rx_k)-1} 
- \frac{\frac{1}{t}\sum_{k=j}^rx_k\exp(\sum_{k=j}^rx_k)}{\{ \frac{1}{t}\exp(\sum_{k=j}^rx_k)-1 \}^2} 
\right) \\
&\quad \cdot\prod_{j\in J}\left( 
\frac{\frac{1}{t}\exp(\sum_{k=j}^rx_k)}{\{ \frac{1}{t}\exp(\sum_{k=j}^rx_k)-1 \}^2}
\right)(\log t)^{\#J} \\
&=\sum_{q\geq0}(\log t)^{q}\left\{
\sum_{\substack{J\subset[r] \\ \#J=q}}
\prod_{j\in [r]\setminus J}\left( 
\frac{1}{\frac{1}{t}\exp(\sum_{k=j}^rx_k)-1} 
- \frac{\frac{1}{t}\sum_{k=j}^rx_k\exp(\sum_{k=j}^rx_k)}{\{ \frac{1}{t}\exp(\sum_{k=j}^rx_k)-1 \}^2} 
\right) \right.\\
&\left. \quad \cdot\prod_{j\in J}\left( 
\frac{\frac{1}{t}\exp(\sum_{k=j}^rx_k)}{\{ \frac{1}{t}\exp(\sum_{k=j}^rx_k)-1 \}^2}
\right) \right\}. 
\intertext{We note that the empty summation is interpreted as $0$.}
&=\sum_{q\geq0}(\log t)^{q}\left\{
\sum_{\substack{J\subset[r] \\ \#J=q}}
\sum_{K\subset[r]\setminus J} (-1)^{\#K}
\prod_{j\in [r]\setminus (J\cup K)}\left( 
\frac{1}{\frac{1}{t}\exp(\sum_{k=j}^rx_k)-1} 
\right) \right.\\
&\left. \quad \cdot
\prod_{j\in K}\left( 
 \frac{\frac{1}{t}\sum_{k=j}^rx_k\exp(\sum_{k=j}^rx_k)}{\{ \frac{1}{t}\exp(\sum_{k=j}^rx_k)-1 \}^2} 
\right)
\prod_{j\in J}\left( 
\frac{\frac{1}{t}\exp(\sum_{k=j}^rx_k)}{\{ \frac{1}{t}\exp(\sum_{k=j}^rx_k)-1 \}^2}
\right) \right\} \\
&=\sum_{q\geq0}(\log t)^{q}\left\{
\sum_{\substack{J\subset[r] \\ \#J=q}}
\sum_{K\subset[r]\setminus J} (-1)^{\#K}
\prod_{j\in [r]\setminus (J\cup K)}\left( 
\frac{1}{\frac{1}{t}\exp(\sum_{k=j}^rx_k)-1} 
\right) \right.\\
&\left. \quad \cdot
\prod_{j\in K}\left( \sum_{k=j}^rx_k \right)
\prod_{j\in J\cup K}\left( 
\frac{\frac{1}{t}\exp(\sum_{k=j}^rx_k)}{\{ \frac{1}{t}\exp(\sum_{k=j}^rx_k)-1 \}^2}
\right) \right\}.
\end{align*}
Using two relations
$$
\frac{1}{\frac{1}{t}e^x-1}
=\sum_{n\geq1}(te^{-x})^n,
\qquad \frac{\frac{1}{t}e^x}{(\frac{1}{t}e^x-1)^2}
=\sum_{n\geq1}n(te^{-x})^n,
$$
we have
\begin{align*}
&\prod_{j=1}^r\left( 
\frac{1}{\frac{1}{t}\exp(\sum_{k=j}^rx_k)-1} 
+ \frac{(\log t-\sum_{k=j}^rx_k)\frac{1}{t}\exp(\sum_{k=j}^rx_k)}{\{ \frac{1}{t}\exp(\sum_{k=j}^rx_k)-1 \}^2} 
\right) \\
&=\sum_{q\geq0}(\log t)^{q}\left\{
\sum_{\substack{J\subset[r] \\ \#J=q}}
\sum_{K\subset[r]\setminus J} (-1)^{\#K}
\prod_{j\in [r]\setminus (J\cup K)}\left( 
\sum_{n_j\geq1}t^{n_j}\exp\left(-n_j\sum_{k=j}^rx_k\right)
\right) \right.\\
&\left. \quad \cdot
\prod_{j\in K}\left( \sum_{k=j}^rx_k \right)
\prod_{j\in J\cup K}\left( 
\sum_{n_j\geq1}n_jt^{n_j}\exp\left(-n_j\sum_{k=j}^rx_k\right)
\right) \right\} \\
&=\sum_{q\geq0}(\log t)^{q}\left\{
\sum_{\substack{J\subset[r] \\ \#J=q}}
\sum_{K\subset[r]\setminus J} (-1)^{\#K}
\sum_{\substack{n_j\geq1 \\ j\in [r]\setminus (J\cup K)}} 
\prod_{j\in [r]\setminus (J\cup K)}t^{n_j}\exp\left(-n_j\sum_{k=j}^rx_k\right)
 \right.\\
&\left. \quad \cdot
\prod_{j\in K}\left( \sum_{k=j}^rx_k \right)
\sum_{\substack{n_j\geq1 \\ j\in J\cup K}} 
\prod_{j\in J\cup K}n_jt^{n_j}\exp\left(-n_j\sum_{k=j}^rx_k\right)
\right\} \\
&=\sum_{q\geq0}(\log t)^{q}\left\{
\sum_{\substack{J\subset[r] \\ \#J=q}}
\sum_{K\subset[r]\setminus J} (-1)^{\#K}
\sum_{n_1,\dots,n_r\geq1}
\prod_{j=1}^rt^{n_j}\exp\left(-n_j\sum_{k=j}^rx_k\right)
\right.\\
&\left. \quad \cdot
\prod_{j\in K}\left( \sum_{k=j}^rx_k \right)
\left( 
\prod_{j\in J\cup K}n_j
\right) \right\} \\
&=\sum_{q\geq0}(\log t)^{q}\left\{
\sum_{\substack{J\subset[r] \\ \#J=q}}
\sum_{K\subset[r]\setminus J} (-1)^{\#K}
\sum_{n_1,\dots,n_r\geq1}
t^{n_1+\cdots+t_r}\prod_{j=1}^r\exp\left(-x_j\sum_{k=1}^jn_k\right)
\right.\\
&\left. \quad \cdot
\prod_{j\in K}\left( \sum_{k=j}^rx_k \right)
\left( 
\prod_{j\in J\cup K}n_j
\right) \right\}.
\end{align*}
Similarly to \cite[(3.18)]{FKMT1}, by putting
\begin{equation}\label{eqn:bK,vecl}
\prod_{j\in K}\left( \sum_{k=j}^rx_k \right)
=\sum_{\vecl\in\N_0^r}b_{K,\vecl}\prod_{j=1}^rx_j^{l_j},
\end{equation}
we get
\begin{align}\label{eqn:integrand of deszeta(vecs:t)}
&\prod_{j=1}^r\left( 
\frac{1}{\frac{1}{t}\exp(\sum_{k=j}^rx_k)-1} 
+ \frac{(\log t-\sum_{k=j}^rx_k)\frac{1}{t}\exp(\sum_{k=j}^rx_k)}{\{ \frac{1}{t}\exp(\sum_{k=j}^rx_k)-1 \}^2} 
\right) \\
&=\sum_{q\geq0}(\log t)^{q}\left\{
\sum_{\substack{J\subset[r] \\ \#J=q}}
\sum_{K\subset[r]\setminus J} (-1)^{\#K}
\sum_{n_1,\dots,n_r\geq1}
t^{n_1+\cdots+t_r}
\left( 
\prod_{j\in J\cup K}n_j
\right)
\right. \nonumber\\
&\left. \quad \cdot
\sum_{\vecl\in\N_0^r}b_{K,\vecl}\prod_{j=1}^rx_j^{l_j}
\prod_{j=1}^r\exp\left(-x_j\sum_{k=1}^jn_k\right)
 \right\}. \nonumber
\end{align}
Assume that $\mathfrak Rs_j$ is sufficiently large for $1\leq j\leq r$.
By using \eqref{eqn:integrand of deszeta(vecs:t)}, we have
\begin{align*}
&\hatdeszeta_r(s_1,\dots,s_r)(t) \\
&=\prod_{k=1}^r\frac{1}{\Gamma(s_k)}
\int_{[0,\infty)^r}\prod_{k=1}^rx_k^{s_k-1}dx_k 
\sum_{q\geq0}(\log t)^{q}\left\{
\sum_{\substack{J\subset[r] \\ \#J=q}}
\sum_{K\subset[r]\setminus J} (-1)^{\#K} \right.\\
&\left. \quad \cdot
\sum_{n_1,\dots,n_r\geq1}
t^{n_1+\cdots+t_r}
\left( 
\prod_{j\in J\cup K}n_j
\right)
\sum_{\vecl\in\N_0^r}b_{K,\vecl}\prod_{j=1}^rx_j^{l_j}
\prod_{j=1}^r\exp\left(-x_j\sum_{k=1}^jn_k\right)
\right\} \\
&=\prod_{k=1}^r\frac{1}{\Gamma(s_k)}
\sum_{q\geq0}(\log t)^{q}\left\{
\sum_{\substack{J\subset[r] \\ \#J=q}}
\sum_{K\subset[r]\setminus J} (-1)^{\#K}
\sum_{n_1,\dots,n_r\geq1}
t^{n_1+\cdots+t_r}
\left( 
\prod_{j\in J\cup K}n_j
\right) \right.\\
&\left. \quad \cdot
\sum_{\vecl\in\N_0^r}b_{K,\vecl}
\prod_{j=1}^r\int_{[0,\infty)}
\exp\left(-x_j\sum_{k=1}^jn_k\right)
x_j^{s_j+l_j-1}dx_j
\right\}.
\intertext{Because we have $\int_{[0,\infty)}
e^{-nx}
x^{s-1}dx=n^{-s}\Gamma(s)$, we calculate}
&=\prod_{k=1}^r\frac{1}{\Gamma(s_k)}
\sum_{q\geq0}(\log t)^{q}\left\{
\sum_{\substack{J\subset[r] \\ \#J=q}}
\sum_{K\subset[r]\setminus J} (-1)^{\#K}
\sum_{n_1,\dots,n_r\geq1}
t^{n_1+\cdots+t_r}
\left( 
\prod_{j\in J\cup K}n_j
\right) \right.\\
&\left. \quad \cdot
\sum_{\vecl\in\N_0^r}b_{K,\vecl}
\prod_{j=1}^r\frac{\Gamma(s_j+l_j)}{(n_1+\cdots+n_j)^{s_j+l_j}}
\right\}.
\end{align*}
Here, we have
\begin{align*}
\prod_{j\in J\cup K}n_j
&=\prod_{j\in J\cup K}\left( \sum_{k=1}^jn_k - \sum_{k=1}^{j-1}n_k \right)
=\sum_{I\subset (J\cup K)\setminus \{1\}}
\prod_{j\in (J\cup K)\setminus I}\left( \sum_{k=1}^jn_k \right) 
\prod_{j\in I}\left( -\sum_{k=1}^{j-1}n_k \right) \\
&=\sum_{I\subset (J\cup K)\setminus \{1\}}(-1)^{\#I}
\prod_{j\in (J\cup K)\setminus I}\left( \sum_{k=1}^jn_k \right) 
\prod_{j+1\in I}\left( \sum_{k=1}^{j}n_k \right),
\end{align*}
so by using this, we get
\begin{align*}
&\hatdeszeta_r(s_1,\dots,s_r)(t) \\
&=\prod_{k=1}^r\frac{1}{\Gamma(s_k)}
\sum_{q\geq0}(\log t)^{q}\left\{
\sum_{\substack{J\subset[r] \\ \#J=q}}
\sum_{K\subset[r]\setminus J} (-1)^{\#K}
\sum_{I\subset (J\cup K)\setminus \{1\}}(-1)^{\#I}
\right.\\
&\left. \quad \cdot
\sum_{\vecl\in\N_0^r}b_{K,\vecl}
\prod_{j=1}^r\Gamma(s_j+l_j)
\sum_{n_1,\dots,n_r\geq1}
\frac{t^{n_1+\cdots+t_r}}{\prod_{j=1}^r(n_1+\cdots+n_j)^{s_j+l_j-\delta_{j\in(J\cup K)\setminus I}-\delta_{j+1\in I}}}
\right\},
\end{align*}
where we use the symbol
$$
\delta_{i\in I}
:=\left\{\begin{array}{ll}
1&(i\in I), \\
0&(i\not\in I),
\end{array}\right.
$$
for $I\subset J\cup K$.
Therefore, we have
\begin{align}\label{eqn:expansion of t-deszeta}
\hatdeszeta_r(s_1,\dots,s_r)(t)
&=\sum_{q\geq0}(\log t)^{q}\left\{
\sum_{\substack{J\subset[r] \\ \#J=q}}
\sum_{K\subset[r]\setminus J}
\sum_{I\subset (J\cup K)\setminus \{1\}} (-1)^{\#K+\#I}
\right.\\
&\left. \quad \cdot
\sum_{\vecl\in\N_0^r}b_{K,\vecl}
\prod_{j=1}^r\frac{\Gamma(s_j+l_j)}{\Gamma(s_j)}
\zeta_r^\shuffle((s_j+l_j-\delta_{j\in(J\cup K)\setminus I}-\delta_{j+1\in I});t)
\right\} \nonumber\\
&=\sum_{q\geq0}(\log t)^{q}\left\{
\sum_{\substack{J\subset[r] \\ \#J=q}}
\sum_{K\subset[r]\setminus J}
\sum_{I\subset (J\cup K)\setminus \{1\}} (-1)^{\#K+\#I}
\right. \nonumber\\
&\left. \quad \cdot
\sum_{\vecl\in\N_0^r}b_{K,\vecl}
\left( \prod_{j=1}^r(s_j)_{l_j} \right)
\zeta_r^\shuffle((s_j+l_j-\delta_{j\in(J\cup K)\setminus I}-\delta_{j+1\in I});t)
\right\}. \nonumber
\end{align}
We next put
$$
\mathcal H((u_j);(v_j))
:=\sum_{\substack{J\subset[r] \\ \#J=q}}
\sum_{K\subset[r]\setminus J}
\sum_{I\subset (J\cup K)\setminus \{1\}} (-1)^{\#K+\#I}
\sum_{\vecl\in\N_0^r}b_{K,\vecl}
\prod_{j=1}^ru_j^{l_j}v_j^{l_j-\delta_{j\in(J\cup K)\setminus I}-\delta_{j+1\in I}}.
$$
Then, by using \eqref{eqn:bK,vecl}, we have
\begin{align*}
&\mathcal H((u_j);(v_j)) \\
&=\sum_{\substack{J\subset[r] \\ \#J=q}}
\sum_{K\subset[r]\setminus J}
\sum_{I\subset (J\cup K)\setminus \{1\}} (-1)^{\#K+\#I}
\prod_{j\in K}\left( \sum_{k=j}^ru_kv_k \right)
\prod_{j=1}^rv_j^{-\delta_{j\in(J\cup K)\setminus I}-\delta_{j+1\in I}} \\
&=\sum_{\substack{J\subset[r] \\ \#J=q}}
\sum_{K\subset[r]\setminus J}
\sum_{I\subset (J\cup K)\setminus \{1\}} (-1)^{\#K+\#I}
\prod_{j\in K}\left( \sum_{k=j}^ru_kv_k \right)
\left(\prod_{j\in (J\cup K)\setminus I} v_j^{-1} \right) 
\left(\prod_{j\in I} v_{j-1}^{-1} \right).
\intertext{Because $(-1)^{\#((J\cup K)\setminus I)}=(-1)^{\#J+\#K-\#I}=(-1)^{\#J}(-1)^{\#K+\#I}$, we calculate}
&=\sum_{\substack{J\subset[r] \\ \#J=q}}(-1)^{\#J}
\sum_{K\subset[r]\setminus J}
\prod_{j\in K}\left( \sum_{k=j}^ru_kv_k \right)
\left\{
\sum_{I\subset (J\cup K)\setminus \{1\}} 
\left(-\prod_{j\in (J\cup K)\setminus I} v_j^{-1} \right) 
\left(\prod_{j\in I} v_{j-1}^{-1} \right)
\right\} \\
&=\sum_{\substack{J\subset[r] \\ \#J=q}}(-1)^{\#J}
\sum_{K\subset[r]\setminus J}
\prod_{j\in K}\left( \sum_{k=j}^ru_kv_k \right)
\prod_{j\in J\cup K} \left(-v_j^{-1} + v_{j-1}^{-1} \right) \\
&=\sum_{\substack{J\subset[r] \\ \#J=q}}(-1)^{\#J}
\prod_{j\in J} \left(-v_j^{-1} + v_{j-1}^{-1} \right)
\sum_{K\subset[r]\setminus J}
\prod_{j\in K}\left\{ \left( \sum_{k=j}^ru_kv_k \right)\left(-v_j^{-1} + v_{j-1}^{-1} \right) \right\}.
\intertext{Here, for any set $I_1\subset I_2$ and for indeterminates $a_j$ ($j\in I_2$), we have
\begin{equation*}\label{eqn:sum prod aj=prod(aj+1)}
\sum_{I_1\subset I_2}\prod_{j\in I_1}a_j
=\prod_{j\in I_2}(a_j+1),
\end{equation*}
where we set the summation of the left-hand side as $1$ for $I_1=\emptyset$.
This relation with $I_1=K$, $I_2=[r]\setminus J$, and $a_j=\left( \sum_{k=j}^ru_kv_k \right)\left(-v_j^{-1} + v_{j-1}^{-1} \right)$ yields}
&=\sum_{\substack{J\subset[r] \\ \#J=q}}
\prod_{j\in J} \left(v_j^{-1} - v_{j-1}^{-1} \right)
\prod_{j\in [r]\setminus J}\left\{ \left( \sum_{k=j}^ru_kv_k \right)\left(-v_j^{-1} + v_{j-1}^{-1} \right)+1 \right\} \\
&=\sum_{\substack{J\subset[r] \\ \#J=q}}
\prod_{j\in J} \left(v_j^{-1} - v_{j-1}^{-1} \right)
\prod_{j\in [r]\setminus J}\left\{ 1- \left( u_jv_j+\cdots+u_rv_r \right)\left(v_j^{-1} - v_{j-1}^{-1} \right) \right\} \\
&=\prod_{j=1}^r\left\{ 1- \left( u_jv_j+\cdots+u_rv_r \right)\left(v_j^{-1} - v_{j-1}^{-1} \right) \right\} \\
&\quad \cdot \sum_{\substack{J\subset[r] \\ \#J=q}}
\prod_{j\in J} \left(v_j^{-1} - v_{j-1}^{-1} \right)
\prod_{j\in J}\left\{ 1- \left( u_jv_j+\cdots+u_rv_r \right)\left(v_j^{-1} - v_{j-1}^{-1} \right) \right\}^{-1} \\
&=\prod_{j=1}^r\left\{ 1- \left( u_jv_j+\cdots+u_rv_r \right)\left(v_j^{-1} - v_{j-1}^{-1} \right) \right\} \\
&\quad \cdot \sum_{\substack{J\subset[r] \\ \#J=q}}
\prod_{j\in J} \frac{v_j^{-1} - v_{j-1}^{-1} }{1- \left( u_jv_j+\cdots+u_rv_r \right)\left(v_j^{-1} - v_{j-1}^{-1} \right)}.
\intertext{By \eqref{eqn:def of mathcalGr} of $\mathcal G_r$, the definition \eqref{eqn:definition of Grk} of $G_{r,k}$ and Lemma \ref{lem:q-th differential of mathcalG}.(2), we have}
&=\mathcal G_r
\sum_{\substack{J\subset[r] \\ \#J=q}}
\prod_{j\in J} G_{r,j} 
=\frac{(-1)^q}{q!}\left( v_r^{-1}\frac{\partial}{\partial u_r} \right)^q\bigl[ \mathcal G_r \bigr],
\end{align*}
that is, we obtain
$$
\mathcal H((u_j);(v_j))
=\frac{(-1)^q}{q!}\left( v_r^{-1}\frac{\partial}{\partial u_r} \right)^q\bigl[ \mathcal G_r \bigr].
$$
Therefore, by using \eqref{eqn:coeff. of q-th differential of mathcalG} and \eqref{eqn:expansion of t-deszeta}, we get
\begin{align*}
\hatdeszeta_r(s_1,\dots,s_r)(t)
&=\sum_{q\geq0}\frac{(-\log t)^q}{q!}\left\{
\sum_{\substack{\vecl=(l_j)\in\N_0^r \\ \vecm=(m_j)\in\Z^r \\ |\vecm|=-q}}
a_{\vecl,\vecm}^r(q)\left( \prod_{j=1}^r(s_j)_{l_j} \right)
\zeta_r^\shuffle(\vecs+\vecm;t)
\right\}.
\end{align*}
Because $\zeta_r^\shuffle(\veck;t)=\Li_{\veck}(t)$ for $\veck\in\Z^r$, by Proposition \ref{prop:explicit expressions of barMq}, we obtain
\begin{align*}
\hatdeszeta_r(k_1,\dots,k_r)(t)
&=\sum_{q\geq0}\frac{(-\log t)^q}{q!}\left\{
\sum_{\substack{\vecl=(l_j)\in\N_0^r \\ \vecm=(m_j)\in\Z^r \\ |\vecm|=-q}}
a_{\vecl,\vecm}^r(q)\left( \prod_{j=1}^r(k_j)_{l_j} \right)
\Li_{\veck+\vecm}(t)
\right\} \\
&
=\sum_{q\geq0}\frac{(-\log t)^q}{q!}Z_q(\veck)
=Z(\veck).
\end{align*}
Hence, we finish the proof.
\end{proof}

\begin{thm}\label{thm:holomorphic of desLi}
For any $k_1,\dots,k_r\in\Z$ and for any $t\in(0,1)$, we have
\begin{equation*}\label{eqn:hypzeta=hypLi}
\hatdeszeta_r(k_1,\dots,k_r)(t)=\desLi(k_1,\dots,k_r)(t).
\end{equation*}
Especially, we have
\begin{equation}\label{eqn:hypLi to deszeta}
\lim_{t \rightarrow 1-0} \desLi(k_1,\dots,k_r)(t)
=\deszeta_r(k_1,\dots,k_r).
\end{equation}
\end{thm}
\begin{proof}
By Theorem \ref{thm:M(veck)=hypLi} and Theorem \ref{thm:hypzeta=M(veck)}, we have
$$
\hatdeszeta_r(k_1,\dots,k_r)(t)=Z(\veck)=\desLi(k_1,\dots,k_r)(t).
$$
By using this equation, we get \eqref{eqn:hypLi to deszeta}:
\begin{align*}
\lim_{t \rightarrow 1-0} \desLi(\veck)(t)
=\lim_{t \rightarrow 1-0}\hatdeszeta_r(\veck)(t)
=\deszeta_r(\veck).
\end{align*}
Here, the second equality holds by Proposition \ref{prop:holomorphic of t-deszeta}.
\end{proof}

\begin{rem}
Proposition \ref{thm:holomorphic of desLi} is a generalization of Proposition \ref{prop:special values of L at non-positive integers}.
\end{rem}

We are now ready to present our main theorem.

\begin{thm}\label{cor:shuffle product of deszeta}

  We define the $\Q$-linear map $\deszeta_\shuffle:\mathcal H\rightarrow \R$ by $
  \deszeta_\shuffle(1):=1$ and
  $$
      \deszeta_\shuffle(j^{k_r}y \cdots j^{k_1}y)
      :=\lim_{t\rightarrow1-0}\desLi(k_1,\dots,k_r)(t),
  $$
  for $k_1,\dots,k_r\in\Z$.
  Then, this map $\deszeta_\shuffle$ forms a $\Q$-algebra homomorphism, i.e.,
  the ``shuffle-type'' formula holds for special values at any integer points of desingularized MZFs.
\end{thm}
\begin{proof}
By Lemma \ref{lem:alg. hom. from mathcalH to mathcalA}, this map $\deszeta_\shuffle$ forms a $\Q$-algebra homomorphism, and the image is nothing but $\deszeta_r(k_1,\dots,k_r)$ by \eqref{eqn:hypLi to deszeta}.
Hence, we obtain the claim.
\end{proof}

\begin{rem}
This theorem is a generalization of Proposition \ref{prop:shuffle type formula of deszeta} and Proposition \ref{prop:shuffle relation of dep1 des MZF}.
\end{rem}

We show some examples of the product of two special values of desingularized MZF.
Especially, we consider 

\begin{exa}
  We calculate the product $\deszeta_1(-k)\deszeta_1(l)$ for $k, l\in\N$.
  By definition of $\shuffle_0$ \eqref{eqn:def of shuffle0}, one can calculate the following:
  For $k$, $l\in\N$, we have
  \begin{equation*}
    d^ky\shuffle_0 j^ly = \sum_{i=0}^{\min\{k,l-1\}}
                                      (-1)^i\binom{k}{i}
                                       d^{k-i}yj^{l-i}y
                                     + (-1)^l\sum_{i=0}^{k-l}
                                          \binom{k-1-i}{l-1}
                                           d^{k-l-i}yd^iy.  
  \end{equation*}
  We note that the empty summation is interpreted as 0. 
  Thus we have
  \begin{align*}
    \deszeta_1(-k)\deszeta_1(l)
     &= \sum_{i=0}^{\min\{k,l-1\}}
          (-1)^i\binom{k}{i}
           \deszeta_2(l-i,-k+i) \\
     &\phantom{=}    
     + (-1)^l\sum_{i=0}^{k-l}
              \binom{k-1-i}{l-1}
               \deszeta_2(-i,-k+l+i).
  \end{align*}
\end{exa}  
  
\begin{exa}  
  We calculate the product $\deszeta_2(l,-k)\deszeta_1(m)$ for $k, l, m\in\N$.
  The following holds for $k$, $l$, $m\in\N$:
  \begin{align*}
    &d^kyj^ly\shuffle_0j^my \\
      &= \sum_{i=0}^{\min\{k,l-1\}}
             \sum_{p=1}^{l+m-i-1}
              (-1)^i\binom{k}{i}
               \left\{ \binom{p-1}{l-1} + \binom{p-1}{m-i-1} \right\}
                d^{k-i}yj^{p}yj^{l+m-i-p}y \\
      &\phantom{=}
             + (-1)^m\sum_{i=0}^{k-m}
                  \binom{m-1+i}{m-1}
                   d^{i}yd^{k-m-i}yj^{l}y.
  \end{align*}
  Thus we have
  \begin{align*}
    &\deszeta_2(l,-k)\deszeta_1(m) \\
      &= \sum_{i=0}^{\min\{k,l-1\}}
             \sum_{p=1}^{l+m-i-1}
              (-1)^i\binom{k}{i}
               \left\{ \binom{p-1}{l-1} + \binom{p-1}{m-i-1} \right\}
                \deszeta_3(l+m-i-p,p,-k+i) \\  
      &\phantom{=}
             + (-1)^m\sum_{i=0}^{k-m}
                  \binom{m-1+i}{m-1}
                   \deszeta_3(l,-k+m+i,-i). 
  \end{align*}
  Based on some numerical experiments with depth 1, we believe that we will be able to prove 
  Proposition \ref{prop:shuffle type formula of deszeta} for $s_1,\dots,s_p\in\Z$ from \eqref{eqn:def of shuffle0} 
  and \eqref{def: definition of two sided ideal mathcalL}.
\end{exa}

\appendix
\section{Explicit formula of $Z_q(\veck)$}
\label{sec:Explicit formula of Mq(veck)}
In this appendix, we give explicit formulae of $Z_q(\veck)$ in Definition \ref{def:definition of Mq} in terms of $\Li_{\veck+\vecm}(t)$ for $\vecm\in\Z^r$.
As a consequence, we find the explicit expression of $Z(\veck)$ which is required  to prove 
Theorem \ref{thm:hypzeta=M(veck)}.

For $r\geq q\geq0$, we define the set of integers $\{a_{\vecl,\vecm}^r(q)\}$ by
\begin{equation}\label{eqn:coeff. of q-th differential of mathcalG}
\left( v_r^{-1}\frac{\partial}{\partial u_r} \right)^q\bigl[ \mathcal G_r \bigr]
=\sum_{\substack{\vecl=(l_j)\in\N_0^r \\ \vecm=(m_j)\in\Z^r \\ |\vecm|=-q}}
a_{\vecl,\vecm}^r(q) \prod_{j=1}^ru_j^{l_j}v_j^{m_j}.
\end{equation}
It is clear that $a_{\vecl,\vecm}^r(0)=a_{\vecl,\vecm}^r$.
For $r\geq k\geq1$, we put
\begin{equation}\label{eqn:definition of Grk}
G_{r,k}:=G_{r,k}((u_j);(v_j))
:=\frac{v_k^{-1}-v_{k-1}^{-1}}{1-(u_kv_k+\cdots+u_rv_r)(v_k^{-1}-v_{k-1}^{-1})}.
\end{equation}
We note that we have
\begin{equation}\label{eqn:differential of Grk}
v_r^{-1}\frac{\partial}{\partial u_r}\bigl[ G_{r,k} \bigr]
=G_{r,k}^2.
\end{equation}
\begin{lem}\label{lem:q-th differential of mathcalG}
The following two hold.
\begin{enumerate}
\item For $q\in[r]$, $\vecl=(l_j)\in\N_0^r$ and $\vecm=(m_j)\in\Z^r$ with $|\vecm|=-q$, we have
$$
a_{\vecl,\vecm}^r(q)=(l_r+1)_q\cdot a_{\vecl^{(-q)},\vecm^{(-q)}}^r.
$$
\item For $q\in[r]$, we have
\begin{equation}\label{eqn:q-th differential of mathcalG}
\left( v_r^{-1}\frac{\partial}{\partial u_r} \right)^q\bigl[ \mathcal G_r \bigr]
=(-1)^qq!\cdot \mathcal G_r
\sum_{\substack{J\subset[r] \\ \#J=q}}\left( \prod_{j\in J}G_{r,j} \right).
\end{equation}
\end{enumerate}
\end{lem}
\begin{proof}
We first prove the claim (1).
By the definition of $\mathcal G_r $, we calculate
\begin{align*}
\left( v_r^{-1}\frac{\partial}{\partial u_r} \right)^q\bigl[ \mathcal G_r \bigr]
&=\sum_{\substack{\vecl=(l_j)\in\N_0^r \\ \vecm=(m_j)\in\Z^r \\ |\vecm|=0}}
a_{\vecl,\vecm}^r \left( \prod_{j=1}^{r-1}u_j^{l_j}v_j^{m_j} \right) (l_r-q+1)_qu_r^{l_r-q}v_r^{m_r-q}.
\intertext{By replacing $l_r-q$ to $l_r$ and $m_r-q$ to $m_r$, we have}
&=\sum_{\substack{\vecl=(l_j)\in\N_0^r \\ \vecm=(m_j)\in\Z^r \\ |\vecm|=-q}}
a_{\vecl^{(-q)},\vecm^{(-q)}}^r \left( \prod_{j=1}^{r-1}u_j^{l_j}v_j^{m_j} \right) (l_r+1)_qu_r^{l_r}v_r^{m_r} \\
&=\sum_{\substack{\vecl=(l_j)\in\N_0^r \\ \vecm=(m_j)\in\Z^r \\ |\vecm|=-q}}
(l_r+1)_q\cdot a_{\vecl^{(-q)},\vecm^{(-q)}}^r \prod_{j=1}^{r}u_j^{l_j}v_j^{m_j}.
\end{align*}
Therefore, by comparing this and \eqref{eqn:coeff. of q-th differential of mathcalG}, we obtain the claim (1).

We next prove the claim (2) by induction on $q$.
We note that, for any rational functions $f_1,\dots,f_r$ in $u_r$, we have
\begin{equation}\label{eqn:differential of product}
\frac{\partial}{\partial u_r}\left[ \prod_{j=1}^r f_j \right]
=\left( \prod_{j=1}^r f_j \right)\sum_{k=1}^r\frac{1}{f_k}\frac{\partial}{\partial u_r}\left[ f_k \right].
\end{equation}
When $q=1$, by using \eqref{eqn:differential of product} and the definition of $\mathcal G_r$, we directly have
\begin{align}\label{eqn:1-th differential of mathcalG}
\left( v_r^{-1}\frac{\partial}{\partial u_r} \right)\bigl[ \mathcal G_r \bigr]
&=v_r^{-1}\mathcal G_r\sum_{k=1}^r
\cdot \frac{-v_r(v_k^{-1}-v_{k-1}^{-1})}{1-(u_kv_k+\cdots+u_rv_r)(v_k^{-1}-v_{k-1}^{-1})} \\
&=-\mathcal G_r\sum_{k=1}^r G_{r,k}
=(-1)^11!\cdot\mathcal G_r\sum_{\substack{J\subset[r] \\ \#J=1}} \left( \prod_{j\in J}G_{r,j} \right). \nonumber
\end{align}
So the equation \eqref{eqn:q-th differential of mathcalG} holds for $q=1$.
Assume the equation \eqref{eqn:q-th differential of mathcalG} holds for $q\,(\geq1)$ or less.
Then, by the induction hypothesis, we have
\begin{align*}
&\left( v_r^{-1}\frac{\partial}{\partial u_r} \right)^{q+1}\bigl[ \mathcal G_r \bigr]
=\left( v_r^{-1}\frac{\partial}{\partial u_r} \right)
\left[ (-1)^qq!\cdot \mathcal G_r
\sum_{\substack{J\subset[r] \\ \#J=q}}\left( \prod_{j\in J}G_{r,j} \right)
\right] \\
&=(-1)^qq! \left( v_r^{-1}\frac{\partial}{\partial u_r} \right)
\left[ \mathcal G_r \right]
\sum_{\substack{J\subset[r] \\ \#J=q}}\left( \prod_{j\in J}G_{r,j} \right)
+  (-1)^qq!\cdot \mathcal G_r
\sum_{\substack{J\subset[r] \\ \#J=q}}
\left( v_r^{-1}\frac{\partial}{\partial u_r} \right)
\left[ \prod_{j\in J}G_{r,j} \right].
\intertext{By applying \eqref{eqn:1-th differential of mathcalG} to the first term and by applying \eqref{eqn:differential of product} to the second term, we calculate}
&=(-1)^{q+1}q!\cdot \mathcal G_r\left\{
\left( \sum_{k=1}^r G_{r,k} \right)
\sum_{\substack{J\subset[r] \\ \#J=q}}\left( \prod_{j\in J}G_{r,j} \right)
-\sum_{\substack{J\subset[r] \\ \#J=q}}
\left( \prod_{j\in J}G_{r,j} \right)
\sum_{k\in J}\frac{v_r^{-1}}{G_{r,k}}\frac{\partial}{\partial u_r}\left[ G_{r,k} \right]
\right\}.
\intertext{By using \eqref{eqn:differential of Grk} in the second term, we get}
&=(-1)^{q+1}q!\cdot \mathcal G_r\left\{
\left( \sum_{k=1}^r G_{r,k} \right)
\sum_{\substack{J\subset[r] \\ \#J=q}}\left( \prod_{j\in J}G_{r,j} \right)
-\sum_{\substack{J\subset[r] \\ \#J=q}}
\left( \prod_{j\in J}G_{r,j} \right)
\sum_{k\in J} G_{r,k}
\right\} \\
&=(-1)^{q+1}q!\cdot \mathcal G_r
\sum_{\substack{J\subset[r] \\ \#J=q}}\left( \prod_{j\in J}G_{r,j} \right)
\left\{
\sum_{k=1}^r G_{r,k} -\sum_{k\in J} G_{r,k}
\right\} \\
&=(-1)^{q+1}q!\cdot \mathcal G_r
\sum_{\substack{J\subset[r] \\ \#J=q}}
\left\{
\left( \prod_{j\in J}G_{r,j} \right)
\sum_{k\in [r]\setminus J} G_{r,k}
\right\} \\
&=(-1)^{q+1}q!\cdot \mathcal G_r
\sum_{\substack{J\subset[r] \\ \#J=q}}
\sum_{k\in [r]\setminus J}
\left\{
\left( \prod_{j\in J}G_{r,j} \right) G_{r,k}
\right\}.
\intertext{By putting $I=J\cup \{k\}$, we have $\#I=\#J+1=q+1$.
So we calculate}
&=(-1)^{q+1}q!\cdot \mathcal G_r
\sum_{\substack{I\subset[r] \\ \#I=q+1}}
\sum_{k\in I}
\left\{
\left( \prod_{j\in I\setminus \{k\}}G_{r,j} \right) G_{r,k}
\right\} \\
&=(-1)^{q+1}q!\cdot \mathcal G_r
\sum_{\substack{I\subset[r] \\ \#I=q+1}}
\sum_{k\in I}
\left( \prod_{j\in I}G_{r,j} \right) \\
&=(-1)^{q+1}q!\cdot \mathcal G_r
\sum_{\substack{I\subset[r] \\ \#I=q+1}}
(q+1) \left( \prod_{j\in I}G_{r,j} \right)
=(-1)^{q+1}(q+1)!\cdot \mathcal G_r
\sum_{\substack{I\subset[r] \\ \#I=q+1}}
\left( \prod_{j\in I}G_{r,j} \right).
\end{align*}
Therefore, the equation \eqref{eqn:q-th differential of mathcalG} holds for $q+1$.
Hence, we obtain the claim (2).
\end{proof}

By the above lemma, we get explicit expressions of $Z_q(\veck)$.
\begin{prop}\label{prop:explicit expressions of barMq}
For any $q\in[r]$, we have
\begin{align}\label{eqn:explicit expressions of barMq}
Z_q(\veck)
=\sum_{\substack{\vecl=(l_j)\in\N_0^r \\ \vecm=(m_j)\in\Z^r \\ |\vecm|=-q}}
a_{\vecl,\vecm}^r(q)
\left(
\prod_{j=1}^r(k_j)_{l_j}
\right)
\Li_{\veck+\vecm}(t).
\end{align}
\end{prop}
\begin{proof}
By Definition \ref{def:definition of Mq}, we have
\begin{align*}
Z_q(\veck)
&=\sum_{i+j=q}(-1)^j\binom{q}{i}D^i\left[ Z_0(\veck^{(j)}) \right].
\intertext{By Proposition \ref{prop:deszeta can be written as sum of shuffle-mzfs} and by $\zeta^\shuffle_r(\veck;t)=\Li_{\veck}(t)$, we calculate}
&=\sum_{i+j=q}(-1)^j\binom{q}{i}D^i\left[
\sum_{\substack{\vecl=(l_j)\in\N_0^r \\ \vecm=(m_j)\in\Z^r \\ |\vecm|=0}}
a_{\vecl,\vecm}^r
\left(
\prod_{a=1}^{r-1}(k_a)_{l_a}
\right) (k_r-j)_{l_r}
\Li_{\veck+\vecm^{(j)}}(t)
\right] \\
&=\sum_{i+j=q}(-1)^j\binom{q}{i}
\sum_{\substack{\vecl=(l_j)\in\N_0^r \\ \vecm=(m_j)\in\Z^r \\ |\vecm|=0}}
a_{\vecl,\vecm}^r
\left(
\prod_{a=1}^{r-1}(k_a)_{l_a}
\right) (k_r-j)_{l_r}
\Li_{\veck+\vecm^{(q)}}(t).
\intertext{By replacing $l_r-q$ to $l_r$ and $m_r-q$ to $m_r$, we get}
&=\sum_{i+j=q}(-1)^j\binom{q}{i}
\sum_{\substack{\vecl=(l_j)\in\N_0^r \\ \vecm=(m_j)\in\Z^r \\ |\vecm|=-q}}
a_{\vecl^{(-q)},\vecm^{(-q)}}^r
\left(
\prod_{a=1}^{r-1}(k_a)_{l_a}
\right) (k_r-j)_{l_r+q}
\Li_{\veck+\vecm}(t) \\
&=\sum_{\substack{\vecl=(l_j)\in\N_0^r \\ \vecm=(m_j)\in\Z^r \\ |\vecm|=-q}}
a_{\vecl^{(-q)},\vecm^{(-q)}}^r
\left(
\prod_{a=1}^{r-1}(k_a)_{l_a}
\right) \left\{
\sum_{i+j=q}(-1)^j\binom{q}{i}(k_r-j)_{l_r+q}
\right\}
\Li_{\veck+\vecm}(t).
\intertext{By using the following Lemma \ref{lem:summation formula for Pochhammer symbols} for $l=l_r$ and $s=k_r$, we have}
&=\sum_{\substack{\vecl=(l_j)\in\N_0^r \\ \vecm=(m_j)\in\Z^r \\ |\vecm|=-q}}
a_{\vecl^{(-q)},\vecm^{(-q)}}^r
\left(
\prod_{a=1}^{r-1}(k_a)_{l_a}
\right) (l_r+1)_q\cdot(k_r)_{l_r}
\Li_{\veck+\vecm}(t) \\
&=\sum_{\substack{\vecl=(l_j)\in\N_0^r \\ \vecm=(m_j)\in\Z^r \\ |\vecm|=-q}}
\left\{ (l_r+1)_q\cdot a_{\vecl^{(-q)},\vecm^{(-q)}}^r \right\}
\left(
\prod_{a=1}^{r}(k_a)_{l_a}
\right) 
\Li_{\veck+\vecm}(t).
\end{align*}
By Lemma \ref{lem:q-th differential of mathcalG}.(1), we obtain \eqref{eqn:explicit expressions of barMq}, hence we finish the proof.
\end{proof}
The following lemma is used in the above proof.
\begin{lem}\label{lem:summation formula for Pochhammer symbols}
For $l,q\geq0$, and $s\in\C$, we have
\begin{equation}\label{eqn:summation formula for Pochhammer symbols}
\sum_{i+j=q}(-1)^j\binom{q}{i}(s-j)_{l+q}=(l+1)_q\cdot(s)_l.
\end{equation}
\end{lem}
\begin{proof}
We prove this by induction on $q\geq0$.
It is clear that the claim holds for $q=0$.
When $q=1$, we calculate the left--and side of \eqref{eqn:summation formula for Pochhammer symbols} as follows.
\begin{align}\label{eqn:summation formula for Pochhammer symbols in q=1}
(s)_{l+1}-(s-1)_{l+1}
=(s)_{l}\cdot(s+l)-(s)_{l}\cdot(s-1)
=(l+1)\cdot(s)_l.
\end{align}
Assume the equation \eqref{eqn:summation formula for Pochhammer symbols} holds for $q-1\,(\geq0)$ or less.
Then, by using Lemma \ref{lem:recurrence formula of the sum of binomial coeff.} for $f(i,j)=(s-j)_{l+q}$, we have
\begin{align*}
&\sum_{i+j=q}(-1)^j\binom{q}{i}(s-j)_{l+q} \\
&=\sum_{i+j=q-1}(-1)^{j}\binom{q-1}{i}(s-j)_{l+q}
+\sum_{i+j=q-1}(-1)^{j+1}\binom{q-1}{i}(s-j-1)_{l+q}.
\intertext{By the induction hypothesis and by using \eqref{eqn:summation formula for Pochhammer symbols in q=1}, we get}
&=(l+2)_{q-1}\cdot(s)_{l+1}-(l+2)_{q-1}\cdot(s-1)_{l+1}
=(l+2)_{q-1}(l+1)\cdot(s)_l
=(l+1)_{q}\cdot(s)_l.
\end{align*}
Hence, we finish the proof.
\end{proof}



\begin{thebibliography}{99}
%
%
\bibitem{AET} Akiyama, S., Egami, S., and Tanigawa, Y., 
Analytic continuation of multiple zeta-functions and their values at non-positive integers, \textit{Acta Arith.}, {\bf 98} (2001), no. 2, 107--116. 

\bibitem{AK} Arakawa, T. and Kaneko, M.,
Multiple zeta values, poly-Bernoulli numbers, and related zeta functions,
\textit{Nagoya Math J.}, {\bf 153} (1999), 189--209.

%


\bibitem{CK} Connes, A., and Kreimer, D., 
Renormalization in quantum field theory and the Riemann-Hilbert problem. I. The Hopf algebra structure of graphs and the main theorem, 2000, \textit{Comm. Math. Phys.}, {\bf 210} (1) 249--273.


\bibitem{EMS1} Ebrahimi-Fard, K., Manchon, D., and Singer, J., 
The Hopf algebra of ($q$)multiple polylogarithms with non-positive arguments, \textit{Int. Math. Res. Notices}, 2017, Vol. {\bf 16}, 4882--4922.


\bibitem{EMT} Essouabri D., Matsumoto K., and Tsumura H.,
Multiple zeta-functions associated with linear recurrence sequences and the vectorial sum formula,
\textit{Canad. J. Math}, Vol {\bf 63} (2), 2011, 241--276. 



\bibitem{FKMT1} Furusho, H., Komori, Y., Matsumoto, K., and Tsumura, H., 
Desingularization of complex multiple zeta-functions, \textit{Amer. J. Math.}, {\bf 139} (2017), 147--173.

\bibitem{FKMT2} Furusho, H., Komori, Y., Matsumoto, K., and Tsumura, H., 
Desingularization of multiple zeta-functions of generalized Hurwitz-Lerch type and evaluation of $p$-adic multiple $L$-functions at arbitrary integers, \textit{RIMS Kokyuroku bessatsu} B68 (2017), 27--66.




\bibitem{Komi} Komiyama, N., 
An equivalence between desingularized and renormalized values of multiple zeta functions at negative integers, Int. Math. Res. Not., no. 2, 551--577, 2019.

\bibitem{Komi2} Komiyama, N.,
Shuffle-type product formulae of desingularized values of multiple zeta-functions, to appear in RIMS Kokyuroku bessatsu.

\bibitem{Komi3} Komiyama, N.,
On shuffle-type functional relations of desingularized multiple zeta-functions,
to appear in J. Number Theory.



		
\bibitem{Man}
Manchon, D., {\it Hopf algebras in renormalization}, 2008, Handbook of algebra, Vol.{\bf 5}, 365--427.
		
		


		
		
\bibitem{Zhao} Zhao, J.,
Analytic continuation of multiple zeta functions, \textit{Proc. Amer. Math. Soc.},
 {\bf 128}  (2000),  no. 5, 1275--1283. 
		
\bibitem{Zhao2} Zhao, J.,
Multiple zeta functions, multiple polylogarithms and their special values, \textit{Series on Number Theory and its Applications}, {\bf 12}. World Scientific Publishing Co. Pte. Ltd., Hackensack, NJ, 2016.
\end{thebibliography}
\end{document}